\documentclass[11pt,a4paper]{article}

\usepackage{authblk}

\usepackage{cite}
\usepackage[english]{babel}
\usepackage{amsmath, amsbsy, amsfonts}
\usepackage{a4wide}
\usepackage{enumitem}
\usepackage{color}

\def\Dscr{\mathcal{D}}
\def\Gscr{\mathcal{G}}
\def\Jscr{\mathcal{J}}
\def\Lscr{\mathcal{L}}
\def\Mscr{\mathcal{M}}
\def\Oscr{\mathcal{O}}
\def\Pscr{\mathcal{P}}
\def\Uscr{\mathcal{U}}
\def\Vscr{\mathcal{V}}

\newcommand{\tond}[1]{{\left(#1\right)}}
\newcommand{\quadr}[1]{{\left[#1\right]}}
\newcommand{\inter}[1]{{\langle#1\rangle}}

\newcommand{\qed}{{\vskip-18pt\null\hfill$\square$\vskip6pt}}
\newcommand{\norm}[1]{\left\|#1\right\|}

\DeclareMathOperator{\Ker}{Ker}
\DeclareMathOperator{\Spec}{Spec}
\DeclareMathOperator{\Range}{Range}
\DeclareMathOperator{\Span}{Span}

\def\CC{\mathbb{C}}
\def\RR{\mathbb{R}}
\def\ZZ{\mathbb{Z}}
\def\TT{\mathbb{T}}
\def\Id{\mathbb I}
\def\naturali{\mathbb{N}}
\def\interi{\mathbb{Z}}
\def\toro{\mathbb{T}}
\def\reali{{\mathbb{R}}}
\def\complessi{{\mathbb{C}}}

\def\epsilon{\varepsilon}
\def\eps{\varepsilon}
\def\vphi{\varphi}

\def\im{{\bf i}}
\def\imunit{{\bf i}}
\def\lie#1{L_{#1}}
\def\rmI{{\rm I}}

\def\Pset{{\cal P}}
\def\realpart{\mathop{\rm Re}\nolimits}
\def\imaginary{\mathop{\rm Im}\nolimits}
\def\prsca#1#2{\langle #1 , #2 \rangle}
\def\parder#1#2{{\partial#1 \over \partial#2}}
\def\frmref{\eqref}
\def\lemref{\ref}

\newtheorem{theorem}{Theorem}[section]
\newtheorem{theorem*}{Theorem}
\newtheorem{lemma}{Lemma}[section]

\newtheorem{proposition}{Proposition}[section]

\newtheorem{remark}{Remark}[section]
\newenvironment{proof}[0]{\noindent{\bf Proof.\space}}{\par\ \hfill$\square$\par\medskip}

\title{On the continuation of degenerate periodic orbits\\ via normal form:
       full dimensional resonant tori}

\author[1]{T. Penati\footnote{{\tt Corresponding
      author, email to tiziano.penati@unimi.it}}}
\author[1]{M. Sansottera}
\author[1]{V. Danesi}

\affil[1]{\small Department of Mathematics, University of Milan, via Saldini 50, 20133 --- Milan, Italy.}

\begin{document}

\maketitle

\begin{abstract}
We reconsider the classical problem of the continuation of degenerate
periodic orbits in Hamiltonian systems.  In particular we focus on
periodic orbits that arise from the breaking of a completely resonant
maximal torus. We here propose a suitable normal form construction
that allows to identify and approximate the periodic orbits which
survive to the breaking of the resonant torus. Our algorithm allows to
treat the continuation of approximate orbits which are at leading
order degenerate, hence not covered by classical averaging methods. We
discuss possible future extensions and applications to localized
periodic orbits in chains of weakly coupled oscillators.
\end{abstract}


\section{Introduction}
\label{sec:introduction}

We consider a canonical system of differential equations with
Hamiltonian
\begin{equation}
H(I,\vphi,\eps) = H_0(I) +\epsilon H_1(I,\vphi)+\eps^2 H_2(I,\vphi)+\ldots\ ,
\label{frm:H-modello}
\end{equation}
where $I\in\Uscr\subset\reali^n$, $\vphi\in\toro^n$ are action-angle
variables and $\epsilon$ is a small perturbative parameter.  The
unperturbed system, $H_0$, is clearly integrable and the orbits, lying
on invariant tori, are generically quasi-periodic.  Besides, if the
unperturbed frequencies satisfy resonance relations, one has periodic
orbits on a dense set of resonant tori.

The KAM theorem ensures the persistence of a set of large measure of
quasi-periodic orbits, lying on non-resonant tori, for the perturbed
system, if $\epsilon$ small enough and a suitable non-degeneracy
condition for $H_0$ is satisfied.

Instead, considering a resonant torus, when a perturbation is added
such a torus is generically destroyed and only a finite number of
periodic orbits are expected to survive.  The location and stability
of the continued periodic orbits are determined by a theorem of
Poincar\'e \cite{PoiI,PoiOU7}, who approached the problem locally:
with an averaging method, he was able to select those isolated
unperturbed solutions which, under a suitable non degeneracy
condition, can be continued by means of an implicit function
theorem. A modern approach has been developed in the seventies by
Weinstein \cite{Wei73b} and Moser \cite{Mos76} using bifurcation
techniques, turning the problem to the investigation of critical
points of a functional on a compact manifold, whose number can be
lower estimated with geometrical methods, like Morse theory.  The
drawback lies in the fact that the method is not at all constructive,
thus does not permit the localization of the periodic orbits on the
torus.  In the same spirit, variational methods which make use of the
mountain pass theorem were developed some years later, by Fadell and
Rabinowitz, under different hypothesis (see Chapter 1 in \cite{Ber07}
for a simplified but extremely clear introduction to this result). In
the last years, some new results appeared
{\cite{VoyI99,MelS05,CorFG13,CorG15}}, which face the problem of the
continuation of {degenerate} periodic orbits and lower dimensional
tori in weakly perturbed Hamiltonian systems.

In this paper we follow the line traced by Poincar\'e and deal with
those cases when the non-degeneracy condition is not fulfilled.  In
particular, under a twist-like condition of the form \eqref{frm:twist}
(see, e.g., \cite{BenGGS84}) and analytic estimates of the
perturbation \eqref{frm:decadimento-iniziale}, we develop an original
normal form scheme, inspired by a recent completely constructive proof
of the classical Lyapunov theorem on periodic orbits
\cite{Giorgilli-2012}, which allows to investigate the continuation of
degenerate periodic orbits.  Precisely, first we identify possible
candidates for the continuation via normal form, then we prove the
existence of a unique solution by using the Newton-Kantorovich method.

\begin{remark}
Let us anticipate a crucial difference with respect to the KAM
normal form algorithm: generically, our normal form procedure turns
out to be divergent.  Actually, a moment's thought suggests that
looking for a convergent normal form which is valid for all the
possible periodic orbits is too much to ask.  The idea is that a
suitably truncated normal form allows to produce the approximated
periodic orbits and the continuation can be performed via contraction
or with a further convergent normal form around a selected periodic
orbit.
\end{remark}

The strength of the present perturbative algorithm is at least
twofold. First, the possibility to construct approximate periodic
solutions at any desired order in $\epsilon$, thus going beyond the
average approximation used in most part of the literature. One of the
few results which represents an improvement with respect to the usual
average method is the one claimed in [MelS05], where a criterion for
the existence of periodic orbits on completely degenerate resonant
tori is proved. In this work the authors, by means of a standard
Liendstedt expansion as the original works of Poincar\'e, are able to
push the perturbation scheme at second order in the small parameter
$\epsilon$. However, the possibility to provide a criterion for the
continuation, although remarkable, is a consequence of the restriction
to completely degenerate cases, like when the Fourier expansion of
$H_1$ with respect to the angle variables does not include a certain
resonance class. In this way, all the partial degeneracies are
excluded.  This limitations is overcome by the normal form
that we propose: indeed, by being able to deal with any degree of
degeneracy, it results more general (also in terms of order of
accuracy), thus including also the above mentioned result.

The formal scheme itself has also a second relevant aspect. Since this
approximation is given by a recursive explicit algorithm, it can be
much useful for numerical applications (see, e.g., \cite{GioSan-2012})
and it is independent on the possibility to conclude the proof with a
contraction theorem.  Second, our approach provides a constructive
normal form that can be extended to a sufficiently general class of
models, including non-linear lattices with next-to-nearest neighbor
interactions, like
  \begin{displaymath}
    H = \sum_{j\in\Jscr}\frac{y_j^2}{2} +\sum_{j\in\Jscr}V(x_j)
    +\epsilon \sum_{l=1}^r\sum_{j\in\Jscr}W(x_{j+l}-x_j)\ ,
  \end{displaymath}
where $V(x)$ is an anharmonic oscillator which allows for action
variable (at least locally, like the Morse potential), and $W(x)$
represents a generic next-to-nearest neighbour (typically linear)
interaction, with $r$ the maximal range of the interaction. This
further generalization would represent a remarkable breakthrough in
the direction of the investigation of degenerate phase-shift
multibreathers and vortexes in one and two dimensional lattices (see,
e.g.,
\cite{Aub97,Ahn98,PelKF05b,KouK09,PelS12,Kou13,PenSPKK16,PenKSKP17}).

In the present work we focus on resonant maximal tori in order to
reduce the technical difficulty to a minimum.  The extension to lower
dimensional tori, that represent the natural extension of the present
work, will be also useful in problems emerging in Celestial Mechanics,
where the persistence of non-resonant lower dimensional tori has been
proved with similar techniques, see,
e.g.,~\cite{SanLocGio-2010,GioLocSan-2014}.

\subsection{Outline of the algorithm and statement of the main results}
\label{sbs:outline}

Consider a completely resonant maximal torus of $H_0$ with unperturbed
frequencies
$$
\hat\omega(I)=\parder{H_0}{I}\ ,\quad
\hbox{such that}\quad
\hat\omega(I) = \omega k\ ,
$$ where $\omega\in\reali$ and $k\in\interi^{n}$. This corresponds to
a suitable choice of the actions $I=I^*$ with non-vanishing
components. From now on, without affecting the generality of the
result, we will assume $k_1=1$: this will simplify the interpretation
of the new variables $\hat q, \hat p$ that we are going to introduce
in a while.

Expanding~\eqref{frm:H-modello} in power series of the
translated actions $J=I-I^*$, one has
\begin{equation*}
  \begin{aligned}
H^{(0)} &= \langle {\hat\omega}, J\rangle
+f_{4}^{(0,0)}(J)+\sum_{l> 2} f_{2l}^{(0,0)}(J)\\
&
+f_{0}^{(0,1)}(\vphi)
+f_{2}^{(0,1)}(J,\vphi )\cr
&
+\sum_{s>1} f_{0}^{(0,s)}(\vphi)
+\sum_{s>1} f_{2}^{(0,s)}(J,\vphi )+\cr
&
+\sum_{s>0}\sum_{l>1}f_{2l}^{(0,s)}(J,\vphi )\ ,
  \end{aligned}
\end{equation*}
where $f_{2l}^{(0,s)}$ is an homogeneous polynomial of degree $l$ in
$J$ and it is a function of order $\Oscr(\epsilon^s)$.

\begin{remark}
  The decision to tie the index $2l$ to terms of degree $l$ in $J$ is
  due to the future extension of the work to lower dimensional
  tori. Indeed, in that case the transversal directions will be
  described in cartesian variables, thus the actions will count for
  two in the total degree.  This is also in agreement with the
  notation adopted in~\cite{GioLocSan-2014}.
\end{remark}

\begin{remark}
  The Hamiltonian~\eqref{frm:H-modello} in most applications has only
  linear terms in the small parameter $\epsilon$, namely $H_{l\geq
    2}\equiv 0$.  Nevertheless, we already consider the general case
  where the perturbation is analytic in the small parameter.  Indeed,
  as it will be clear from the normal form procedure, starting from
  the first normalization step we immediately introduce the whole
  series expansion in $\epsilon$.
\end{remark}

We define the $(n-1)$-dimensional resonant module
$$
\Mscr_{\omega} = \Bigl\{ h\in\interi^{n}: \langle
{{\hat\omega}},h\rangle =0\Bigr\}\ 
$$ and introduce the resonant variables $\hat p$,$\hat q$ in place of
$J,\vphi$.  In particular, the pair of conjugate variables
$\hat{p}_1,\hat{q}_1$ describe the periodic orbit, while the pairs
$\hat{p}_j,\hat{q}_j$, $j=2,\ldots,n$, represent the transverse
directions. The canonical change of coordinates is built with an
unimodular matrix (see Lemma~2.10 in~\cite{Gio03}) which shows
that\footnote{{This follows from the assumption $k_1=1$. Indeed, in
    this case that the resonant vector defining the phase differences
    $\hat q_j = k_j \phi_1 - \phi_j$ are a basis for the resonant
    modulus $\Mscr_{\omega}$.}} the new angles $\hat{q}_{j}$,
$j=2,\ldots,n$, are the phase differences with respect to the
\emph{true} angle of the periodic orbit, $\hat{q}_1$, {and that $\hat
  p_1$ is given by $\hat p_1 = \inter{k, J}$}.

Introducing the convenient notations $\hat{p}=(p_1,p)$,
$\hat{q}=(q_1,q)$ with $p_1=\hat{p}_1$,
$p=(\hat{p}_2,\ldots,\hat{p}_n)$ and correspondingly for $q_1$ and
$q$, the Hamiltonian can be written in the form
\begin{equation}
  \begin{aligned}
    H^{(0)} &= \omega p_1 +f_{4}^{(0,0)}(p_1,p)+\sum_{l> 2}
    f_{2l}^{(0,0)}(p_1,p )\cr &\quad +f_{0}^{(0,1)}(q_1,q)
    +f_{2}^{(0,1)}(p_1,p,q_1,q )\cr &\quad +\sum_{s>1}
    f_{0}^{(0,s)}(q_1,q) +\sum_{s>1} f_{2}^{(0,s)}(p_1,p,q_1,q )\cr
    &\quad
    +\sum_{s>0}\sum_{{l>1}}f_{2l}^{(0,s)}(p_1,p,q_1,q)
    \label{frm:H(0)}
  \end{aligned}
\end{equation}
where $f_{2l}^{(0,s)}$ is an homogeneous polynomial of degree $l$ in
$\hat{p}$ and it is a function of order $\Oscr(\epsilon^s)$.

We consider the extended complex domains $\Dscr_{\rho,\sigma} =
\Gscr_{\rho} \times \TT^{n}_{\sigma}$, precisely
$$
\begin{aligned}
\Gscr_{\rho}&=\big\{\hat{p}\in\CC^{n}:\max_{1\le j\le n}|\hat{p}_j|<\rho\big\}\ ,\\
\TT^{n}_{\sigma}&=\big\{\hat{q}\in\CC^{n}:\realpart
  \hat{q}_j\in\TT,\break\ \max_{1\le j\le n}|\imaginary
  \hat{q}_j|<\sigma\big\}\ ,
\end{aligned}
$$
and introduce the distinguished classes of functions $\Pset_{2l}\,$,
with integers $l$, which can be written as a Fourier-Taylor expansion
\begin{equation}
g(\hat{p},\hat{q}) =
\sum_{{i\in\naturali^{n} \atop |i|= l}}\,\sum_{k\in\ZZ^n}
g_{i,k}
\,\hat{p}^{i} e^{\imunit \langle k,\, \hat{q}\rangle} \ ,
\label{frm:funz}
\end{equation}
with coefficients $g_{i,k}\in\complessi$.
We also set $\Pset_{-2}=\{0\}$.

Let us consider a generic analytic function $g\in\Pscr_{2l}$,
$g:\Dscr_{\rho,\sigma}\to\complessi$, we define the weighted Fourier
norm
\begin{equation*}
\|g\|_{\rho,\sigma}=
\sum_{{i\in\naturali^{n} \atop |i| = l}}
\sum_{{\scriptstyle{k\in\ZZ^{n}}}}
|g_{i,k}| \rho^l e^{|k|\sigma}\ .
\end{equation*}
Hereafter, we use the shorthand notation $\|\cdot\|_{\alpha}\,$ for
$\|\cdot\|_{\alpha(\rho,\sigma)}\,$.

We state here our main result concerning the normal form part

\begin{proposition}\label{pro:forma-normale-1}
Consider a Hamiltonian $H^{(0)}$ expanded as in~\frmref{frm:H(0)}
that is analytic in a domain $\Dscr_{\rho,\sigma}$. Let us assume that
\begin{enumerate}[label=(\alph*)]
\item
  there exists a positive constant $m$ such that for every
  $v\in\reali^{n}$ one has
  \begin{equation}
    m \sum_{i=1}^n |v_i| \leq \sum_{i=1}^n |\sum_{j=1}^n C_{ij}
    v_j|\ ,\quad\hbox{where}\quad C_{ij}=\frac{\partial^2
      f_4^{(0,0)}}{\partial \hat{p}_i \partial \hat{p}_j}\ ;
    \label{frm:twist}
  \end{equation}
  
\item the terms appearing in the expansion of the Hamiltonian satisfy
  \begin{equation}
  \|f_l^{(0,s)}\|_1 \leq \frac{E}{2^l} \epsilon^s\ ,\quad\hbox{with}\quad E>0.
  \label{frm:decadimento-iniziale}
  \end{equation}
\end{enumerate}
Then, for every positive integer $r$ there is a positive
$\epsilon_{r}^*$ such that for $0\leq \epsilon < \epsilon_{r}^*$ there
exists an analytic canonical transformation $\Phi^{(r)}$ satisfying
\begin{equation}
\Dscr_{\frac{1}{4}(\rho,\sigma)} \subset
\Phi^{(r)}\Bigl(\Dscr_{\frac{1}{2}(\rho,\sigma)}\Bigr) \subset
\Dscr_{\frac{3}{4}(\rho,\sigma)}
\label{frm:domini}
\end{equation}
such that the Hamiltonian $H^{(r)}= H^{(0)} \circ \Phi^{(r)}$ is in
normal form up to order $r$, namely
\begin{equation}
  \begin{aligned}
    H^{(r)}(p_1,p,q_1,q;q^*) &= \omega p_1
    +f_{4}^{(r,0)}(p_1,p)+\sum_{l> 2} f_{2l}^{(r,0)}(p_1,p)\cr &\quad
    +\sum_{s=1}^{r} f_{0}^{(r,s)}(q;q^*)+\sum_{s=1}^{r}
    f_{2}^{(r,s)}(p_1,p,q;q^*)\cr &\quad +\sum_{s>r}
    f_{0}^{(r,s)}(q_1,q;q^*) +\sum_{s>r}
    f_{2}^{(r,s)}(p_1,p,q_1,q;q^*)\cr &\quad
    +\sum_{s>0}\sum_{{l>1}}f_{2l}^{(r,s)}(p_1,p,q_1,q;q^*)\ ,
    \label{frm:H(r)}
  \end{aligned}
\end{equation}
where $q^*$ is a fixed but arbitrary parameter and
$f_{2l}^{(r,s)}\in\Pscr_{2l}$ is a function of order $\Oscr(\epsilon^s)$.
Moreover, for $q=q^*$ one has
\begin{equation}
  \sum_{s=1}^{r} f_{2}^{(r,s)}(p_1,p,q^*;q^*) = 0\ .
  \label{frm:nf-qstar}
\end{equation}
\end{proposition}
A comment on the structure of the normal form is in order.  We aim to
continue a generic unperturbed periodic orbit $p_1=0$, $q_1=q_1(0) +
\omega t$, $p=0$, $q=q^*$, thus we look for a normal form which is
able to select those phase shifts, $q^*$, which represent good
candidates for the continuation.  The Hamiltonian is said to be in
normal form up to order $r$ if the constant and linear terms in the
actions are averaged (up to order $r$) with respect to the fast angle,
$q_1$, and if, for a fixed but arbitrary $q^*$, the linear terms in
the action fulfill~\eqref{frm:nf-qstar}.

The crucial point is that the Hamilton equations associated to the
truncated normal form, i.e., neglecting term of order
{$\Oscr(\epsilon^{r+1})$}, once evaluated at
$(\hat{p}=0,q=q^*)$, read
\begin{equation*}
\dot{p}_1 = 0\ ,  \qquad
\dot{q}_1 = \omega\ ,  \qquad
\dot{p} = -\sum_{s=1}^r \nabla_{q} f_0^{(r,s)}\ ,  \qquad
\dot{q} = 0\ .
\end{equation*}
Hence, if
\begin{equation}
\label{frm:qstar}
\sum_{s=1}^r \nabla_q {f_0^{(r,s)}}\big|_{q=q^*} = 0\ ,  
\end{equation}
then $p_1=0$, $q_1=q_1(0)$, $p=0$, $q=q^*$ is the initial datum of a
periodic orbit with frequency $\omega$ for the truncated normal form.
Considering the whole system given by $H^{(r)}$, the initial datum provides
an \emph{approximate} periodic orbit with frequency $\omega$, {which
  turns out to be a relative equilibrium of the truncated
  Hamiltonian}.  In order to provide a precise definition of
\emph{approximate periodic orbit} we
introduce the $T$-period map
$\Upsilon:\Uscr(q^*,0)\subset\reali^{2n-1}\to
\Vscr(0,q^*)\subset\reali^{2n-1}$, a smooth function of the $2n-1$
variables $(q,\hat{p})$, parametrized by the initial phase $q_1(0)$
and the small parameter $\epsilon$, precisely
\begin{equation}
  \label{frm:Ups}
\Upsilon(q(0),\hat{p}(0);\epsilon,q_{1}(0))=
  \begin{pmatrix}
    F(q(0),\hat{p}(0);\epsilon,q_{1}(0))\\
    G(q(0),\hat{p}(0);\epsilon,q_{1}(0))
  \end{pmatrix}  =
  \begin{pmatrix}
  \hat{q}(T)-\hat{q}(0)-\Lambda T\\
  \frac{1}{\epsilon}({{p}(T)-{p}(0)})
  \end{pmatrix} \ ,
\end{equation}
with $\Lambda=(\omega,0)\in\reali^{n}$. The map $\Upsilon$ represents the
$T$-flow of the $n-1$ actions $p$ and of the $n$ angles
$\hat{q}$ for the Hamiltonian $H^{(r)}$.

Let us stress that $p_1=0$, $q_1=q_1$, $p=0$, $q=q^*$ corresponds to a
periodic orbit for the truncated normal form, thus it is evident that
$\Upsilon(q^*,0;\eps,q_1(0))$ is of order\footnote{The actions $p$
  have been rescaled by $\epsilon$ in $\Upsilon$, hence only $G$ is of
  order $\Oscr(\epsilon^{r+1})$ while $F$ is of order
  $\Oscr(\epsilon^{r})$.}  $\Oscr(\eps^{r})$.  Thus, a true periodic
orbit, close to the approximate one, is identified by an initial datum
$(q^*_{\rm p.o.},\hat{p}_{\rm p.o.})\in\Uscr(0,q^*)$ such that
\begin{equation*}
  \Upsilon(q^*_{\rm p.o.},\hat{p}_{\rm p.o.};\epsilon,q_1(0))= 0\ .
\end{equation*}

In order to prove the existence of a unique solution $q^*=q^*_{\rm
  p.o.}$, $\hat{p}=\hat{p}_{\rm p.o.}$, $q_1=q_1(0)$, close enough to
the approximate one, we apply the Newton-Kantorovich algorithm.
Therefore we need to ensure that the Jacobian matrix (with respect to
the initial datum)
\begin{equation}
  \label{frm:M(epsilon)}
 M(\eps) = D_{\hat{p}(0),q^*(0)}\Upsilon(q^*,0;\eps, q_1(0)) 
\end{equation}
is invertible and its eigenvalues are not too small with respect to
$\epsilon^r$.

We state here the main result concerning the continuation of the periodic orbits
\begin{theorem}
  \label{teo:forma-normale-r}
  Consider the map $\Upsilon$ defined in~\eqref{frm:Ups} in a
  neighbourhood of the torus $\hat{p}=0$ and let
  $({q^*}(\epsilon),0)$, with ${q^*}(\eps)$
  satisfying~\eqref{frm:qstar}, an approximate zero of $\Upsilon$,
  namely
  \begin{displaymath}
    \norm{\Upsilon({q^*}(\epsilon),0;\epsilon,q_1(0))}\leq C_1
    \epsilon^r\ ,
  \end{displaymath}
  where $C_1$ is a positive constant just depending on $\Uscr$. Assume
  that the matrix $M(\eps)$ defined in~\eqref{frm:M(epsilon)} is invertible and
  its eigenvalues satisfy
  \begin{equation}
    \label{frm:stima-autovalori}
    |\lambda|\geq \eps^\alpha\ ,\qquad
    \hbox{for}\quad\lambda\in \mathrm{spec} \tond{M(\eps)}\quad
    \hbox{with}\quad 2\alpha<r\ .
  \end{equation}
  Then, there exist $C_0>0$ and $\eps^*>0$ such that for any
  $0\leq\epsilon<\epsilon^*$ there exists a unique
  $({q^*}_{\rm p.o.}(\epsilon),\hat{p}_{\rm p.o.}(\epsilon))\in\Uscr$ which solves
\begin{equation}
\label{frm:estim.sol}
    \Upsilon({q^*}_{\rm p.o.},\hat{p}_{\rm p.o.};\epsilon,q_1(0))=0
    \ ,\qquad\qquad \norm{({q^*}_{\rm
        p.o.},\hat{p}_{\rm p.o.})-({q^*},0)}\leq C_0\eps^{r-\alpha}\ .
\end{equation}
\end{theorem}
Before entering the technical part of the paper, let us add some more
considerations.  First, as already remarked, the above Theorem
generalizes an old and classical result by Poincar\'e, whose idea was
to average the perturbation $H_1$ with respect to the flow to the
unperturbed periodic solution, where only the fast angle $q_1$
rotates.  The candidates for the continuation, ${q}^*$, were the
non-degenerate relative extrema on the torus $\mathbb{T}^{n-1}$ of the
averaged Hamiltonian $\inter{H_1}_{q_1}$, namely
\begin{displaymath}
  \nabla_{{q}}\inter{H_1}_{q_1}=0\ ,\qquad |D^2_{{q}}
  \inter{H_1}_{q_1}|\not=0\ .
\end{displaymath}
The result of Poincar\'e actually corresponds to the construction of
the first order normal form together with a non-degeneracy assumption
on the $\epsilon$-independent version of \eqref{frm:qstar}, precisely
\begin{equation}
  \label{frm:qstar.Poi}
\nabla_{q} f_0^{(1,1)} = 0\ ,\qquad |D^2_{{q}} f_0^{(1,1)}|\not=0\ .
\end{equation}
In such a case, due to the simplified form of $\Upsilon$, the solution
$(\hat{p}_{\rm p.o.},q^*_{\rm p.o.})$ can be obtained via implicit
function theorem in a neighborhood of the approximate initial datum
$({0},{q}^*)$, being ${q}^*$ a solution of the first
of~\eqref{frm:qstar.Poi}, independent of $\epsilon$.  Hence, our
high-order normal form construction becomes a necessary way in order
to deal with \emph{degenerate cases}, where solutions of
\eqref{frm:qstar.Poi} are not isolated and appear as $d$-parameter
families, thus leading to $|D^2_{{q}} f_0^{(1,1)}|=0$.

For instance, in the application presented in
Section~\ref{sec:applic}, {the solutions of \eqref{frm:qstar.Poi} show
  up as one parameter families ${q}^*(s)$}.  Actually,
solving~\eqref{frm:qstar} (with $r\geq 2$) in place
of~\eqref{frm:qstar.Poi} allows to isolate true candidates for the
continuation.  Let us also remark that our scheme provides a refined
averaged Hamiltonian which allows to treat the totally degenerate
case, i.e., $\nabla_{{q}} f^{(1,1)}_0\equiv 0$. In particular, the
results presented in~\cite{MelS05} by means of Liendstedt perturbation
scheme can be obtained as special cases.

The paper is organized as follows.  In Section~\ref{sec:normal-form}
we detail the normal form algorithm together with the quantitative
estimates The proof of Theorem~\ref{teo:forma-normale-r} is reported
in Section~\ref{sec:teo}.  Section~\ref{sec:applic} provides a
simplified version of Theorem~\ref{teo:forma-normale-r}, namely
Theorem~\ref{teo:forma-normale-2}, for one parameter families of
solutions of \eqref{frm:qstar.Poi}, under the assumption that only the
second normal form step is enough to improve the accuracy of the
approximate periodic orbit.  Moreover, a pedagogical example inspired
by the problem of degenerate vortexes in a squared lattice dNLS model
is presented at the end of Section~\ref{sec:applic}.
Appendices~\ref{app:est} and~\ref{app:fixedpoint} include the
technicalities related to the normal form estimates and the
Newton-Kantorovich method, respectively.


\section{Normal formal algorithm and analytical estimates}
\label{sec:normal-form}

This Section is devoted to the formal algorithm that takes a
Hamiltonian~\eqref{frm:H(0)} and brings it into normal form up to an
arbitrary, but finite, order $r$.  We include all the (often tedious)
formul{\ae} that will be used in order to estimate the terms appearing
in the normalization process. We use the formalism of Lie series and
Lie transforms (see, e.g.,~\cite{Gro60} and~\cite{Gio03} for a
self-consistent introduction).

The transformation at step $r$ is generated via composition of two Lie
series of the form
$$
\exp(L_{\chi_2^{(r)}}) \circ \exp(L_{\chi_0^{(r)}})\ ,
$$
where
\begin{equation}
\chi^{(r)}_{0} = X^{(r)}_{0} + \prsca{\zeta^{(r)}}{\vphi}\ ,
\label{frm:chi0-r}
\end{equation}
with $\zeta^{(r)}\in\reali^n$ and $X^{(r)}_{0}\in\Pscr_0$,
$\chi^{(r)}_{2}\in\Pscr_2$ are of order $\Oscr(\epsilon^r)$.  Here, as
usual, we denote by $L_g\cdot$ the Poisson bracket $\{g,\cdot\}$.  The
functions $\chi_0^{(r)}$ and $\chi_2^{(r)}$ are unknowns to be
determined so that the transformed Hamiltonian is in normal form up to
order $r$.

The relevant algebraic property of the $\Pscr_{\ell}$ classes of
function is stated by the following
\begin{lemma}
  \label{lem:poisson}
  Let $f\in\Pscr_{s_1}$ and $g\in\Pscr_{s_2}$, then
  $\{f,g\}\in\Pscr_{s_1+s_2-2}$.
 \end{lemma}

\noindent The straightforward proof is left to the reader.

The starting Hamiltonian has the form
\begin{equation}
\begin{aligned}
  H^{(0)} &= \omega p_1
+\sum_{s\geq0}\sum_{l>1}f_{2l}^{(0,s)}\cr
&\quad
+\sum_{s\geq1} f_{0}^{(0,s)}
+\sum_{s\geq1} f_{2}^{(0,s)}\ ,
\end{aligned}
\label{frm:H(0)-f}
\end{equation}
where $f_{2l}^{(0,s)}\in\Pset_{2l}$ and is of order
$\Oscr(\epsilon^s)$.

We now describe the generic $r$-th normalization step, starting from the Hamiltonian in normal form up to order $r-1$, $H^{(r-1)}$, namely
\begin{equation}
\begin{aligned}
H^{(r-1)} &= \omega p_1
+\sum_{s<r} f_{0}^{(r-1,s)}
+\sum_{s<r} f_{2}^{(r-1,s)}\cr
&\quad
+f_{0}^{(r-1,r)}
+f_{2}^{(r-1,r)}\cr
&\quad
+\sum_{s>r} f_{0}^{(r-1,s)}
+\sum_{s>r} f_{2}^{(r-1,s)}\cr
&\quad
+\sum_{s\geq 0}\sum_{l>1}f_{2l}^{(r-1,s)}\ ,
\label{frm:H(r-1)-f}
\end{aligned}
\end{equation}
where $f_{2l}^{(r-1,s)}\in\Pset_{2l}$ is of order
$\Oscr(\epsilon^s)$; $f_{0}^{(r-1,s)}$ and
$f_{2}^{(r-1,s)}$ for $1\leq s < r$ are in normal form.

\subsection{First stage of the normalization step}\label{sbs:step1}
Our aim is to put the term $f_{0}^{(r-1,r)}$ in normal form and to
keep fixed the harmonic frequencies of the selected resonant torus.
We determine the generating function $\chi^{(r)}_{0}=X^{(r)}_{0} +
\prsca{\zeta^{(r)}}{\hat q}$ by solving the homological equations
\begin{equation*}
  \begin{aligned}
\lie{X^{(r)}_{0}} \omega p_1
+ f_{0}^{(r-1,r)}  =\langle f_{0}^{(r-1,r)} \rangle_{q_1} \ ,\\
\lie{\prsca{\zeta^{(r)}}{{\hat q}}} f_{4}^{(0,0)} +
\Bigl\langle
f_{2}^{(r-1,r)}
\Bigr|_{q=q^*}\Bigr\rangle_{q_1}
= 0\ .
  \end{aligned}
\end{equation*}
Considering the Taylor-Fourier expansion
\begin{equation*}
f_{0}^{(r-1,r)}({\hat q})= \sum_{k}
c_{0,,k}^{(r-1)}\exp(\imunit \prsca{k}{{\hat q}})\ ,
\end{equation*}
we readily get
\begin{equation*}
X^{(r)}_{0}({\hat q}) =\sum_{k_1\neq0}
\frac{c_{0,k}^{(r-1)}}{\imunit {k_1}{\omega}}
\exp(\imunit \prsca{k}{{\hat q}})\ .
\end{equation*}
The translation vector, $\zeta^{(r)}$, is determined by solving the
linear system
\begin{equation}
\sum_j C_{ij} \zeta_j^{(r)} +
\parder{}{{\hat p_i}}\Bigl\langle f_{2}^{(r-1,r)}
\Bigr|_{q=q^*}\Bigr\rangle_{q_1} = 0\ .
\label{frm:traslazione}
\end{equation}
This translation, which involves the linear term in the actions
$f_{2}^{(r-1,r)}$, allows to keep fixed the frequency $\omega$ and
kills the small transversal frequencies in the angles $q$.

The transformed Hamiltonian is computed as
\begin{equation*}
H^{(\rmI;r-1)} = \exp\left(\lie{\chi_{0}^{(r)}}\right)H^{(r-1)}
\end{equation*}
and has a form similar to~\eqref{frm:H(r-1)-f}, precisely
\begin{equation}
\begin{aligned}
H^{(\rmI;r-1)}&=
\exp\Bigl(\lie{\chi^{(r)}_{0}}\Bigr) H^{(r-1)} = \cr
&= \omega p_1
+\sum_{s<r} f_{0}^{(\rmI;r-1,s)}
+\sum_{s<r} f_{2}^{(\rmI;r-1,s)}\cr
&\quad
+f_{0}^{(\rmI;r-1,r)}
+f_{2}^{(\rmI;r-1,r)}\cr
&\quad
+\sum_{s>r} f_{0}^{(\rmI;r-1,s)}
+\sum_{s>r} f_{2}^{(\rmI;r-1,s)}\cr
&\quad
+\sum_{s\geq 0}\sum_{l>1}f_{2l}^{(\rmI;r-1,s)}\ .
\label{frm:H(I;r-1)-f}
\end{aligned}
\end{equation}
The functions $f_{2l}^{(\rmI;r-1,s)}$ are recursively defined as
\begin{equation}
\begin{aligned}
f_{0}^{(\rmI;r-1,r)} &= \bigl\langle f_{0}^{(r-1,r)} \bigr\rangle_{q_1}\ ,
\cr
f_{2}^{(\rmI;r-1,r)} &= f_{2}^{(r-1,r)} - \bigl\langle f_{2}^{(r-1,r)}(q^*)\bigr\rangle_{q_1} + \lie{X_0^{(r)}} f_4^{(0,0)}\ ,
\cr
\noalign{\smallskip}
f_{2l}^{(\rmI;r-1,s)} &= \sum_{j=0}^{\lfloor s/r\rfloor}
\frac{1}{j!} \lie{\chi^{(r)}_{0}}^{j} f^{(r-1,s-jr)}_{2l+2j}\ ,
\label{frm:f^(I;r-1,s)}
\end{aligned}
\end{equation}
with $f_{2l}^{(\rmI;r-1,s)}\in\Pset_{2l}$.

\subsection{Second stage of the normalization step}
\label{sbs:step2}
We now put $f_{2}^{(\rmI;r-1,r)}$ in normal form, by averaging with
respect to the fast angle $q_1$. This is necessary to avoid small
oscillations of $q$ around $q^*$.  We
determine the generating function $\chi^{(r)}_2$ by solving the
homological equation
\begin{equation*}
\lie{\chi^{(r)}_{2}} {\omega}{p_1}
+f_{2}^{(\rmI;r-1,r)}
= \left\langle
f_{2}^{(\rmI;r-1,r)}
\right\rangle_{q_1} \ .
\end{equation*}
Considering again the Taylor-Fourier expansion
\begin{equation*}
f_{2}^{(\rmI;r-1,r)}(\hat p,{\hat q})
=\sum_{|l|=1\atop k}\,
c_{l,k}^{(\rmI;r-1)}\hat p^{l}\exp(\imunit\prsca{k}{{\hat q}})
\end{equation*}
we get
\begin{equation*}
\chi_{2}^{(r)}(\hat p,{\hat q})
=\sum_{|l|=1\atop k_1\neq0}\,
\frac{c_{l,k}^{(\rmI;r-1)}\hat p^{l}\exp(\imunit\prsca{k}{{\hat q}})}{\imunit k_1 \omega}\ .
\end{equation*}

The transformed Hamiltonian is computed as
\begin{equation*}
H^{(r)} = \exp\left(\lie{\chi_{2}^{(r)}}\right)H^{(\rmI;r-1)}
\end{equation*}
and is given the form~\eqref{frm:H(r-1)-f}, replacing
the upper index $r-1$ by $r\,$, with 
\begin{equation}
\begin{aligned}
f_{2}^{(r,r)} &= \langle f_{2}^{(\rmI;r-1,r)}\rangle_{q_1}\ ,
\cr
f_{2}^{(r,jr)} &= \frac{1}{(j-1)!}\lie{\chi^{(r)}_{2}}^{j-1}
\left(
\frac{1}{j}\langle f_{2}^{(\rmI;r-1,r)}\rangle_{q_1}
+\frac{j-1}{j}f_{2}^{(\rmI;r-1,r)}
\right)
\cr
&\quad+
\sum_{j=0}^{\lfloor s/r\rfloor-2}
\frac{1}{j!} \lie{\chi^{(r)}_{2}}^{j} f^{(\rmI;r-1,s-jr)}_{2}\ ,
\cr
\noalign{\smallskip}
f_{2l}^{(r,s)} &= \sum_{j=0}^{\lfloor s/r\rfloor}
\frac{1}{j!} \lie{\chi^{(r)}_{2}}^{j} f^{(\rmI;r-1,s-jr)}_{2l}\ .
\label{frm:f^(r,s)}
\end{aligned}
\end{equation}

\subsection{Analytic estimates}\label{sbs:stime}

In order to translate our formal algorithm into a recursive scheme of
estimates on the norms of the various functions, we need to introduce
a sequence of restrictions of the domain so as to apply Cauchy's
estimate.  Having fixed $d\in\reali$, $0 < d \leq 1/4$, we consider a
sequence ${\delta}_{r\geq1}$ of positive real numbers satisfying
\begin{equation}
  \delta_{r+1} \leq \delta_r\ ,\quad \sum_{r\geq1} \delta_r \leq \frac{d}{2}\ .
  \label{frm:delta_r}
\end{equation}
Moreover, we introduce a further sequence ${d}_{r\geq0}$ of real
numbers recursively defined as
\begin{equation}
  d_0 = 0\ ,\quad d_r = d_{r-1} + 2\delta_r\ .
  \label{frm:d_r}
\end{equation}

In order to precisely state the iterative Lemma, we need to introduce
the quantities $\Xi_r$, parametrized by the index $r$, as
\begin{equation}
  \label{frm:Xi-r}
  \Xi_r=\max\left(1,
  \frac{E}{\omega\delta_r^2\rho\sigma}+  \frac{eE}{4m\delta_r\rho^2},
  2 + \frac{E}{2e\omega\delta_r\rho\sigma},
  \frac{E}{4\omega\delta_r^2\rho\sigma}
  \right)
  \ .
\end{equation}
The number of terms in formul{\ae} \eqref{frm:f^(I;r-1,s)} and
\eqref{frm:f^(r,s)} is controlled by the two sequences
$\{\nu_{r,s}\}_{r\ge 0\,,\,s\ge 0}$ and
$\{\nu_{r,s}^{(\text{I})}\}_{r\ge 1\,,\,s\ge 0}$:
\begin{equation} \label{frm:nu-sequence}
\vcenter{\openup1\jot \halign{
    $\displaystyle\hfil#$&$\displaystyle{}#\hfil$&$\displaystyle#\hfil$\cr
    \nu_{0,s} &= 1 &\quad\hbox{for } s\geq 0\,, \cr
    \nu_{r,s}^{(\text{I})} &= \sum_{j=0}^{\lfloor s/r \rfloor}
    \nu_{r-1,r}^{j}\nu_{r-1,s-jr} &\quad\hbox{for } r\geq 1\,,\ s\geq
    0\,, \cr \nu_{r,s} &= \sum_{j=0}^{\lfloor s/r \rfloor}
    (3\nu_{r-1,r})^{j}\nu_{r,s-jr}^{(\text{I})}
    &\quad\hbox{for } r\geq 1\,,\ s\geq 0\,.  \cr }}
\end{equation}
Let us stress that when $s<r$, the above simplify as
\begin{displaymath}
  \nu_{r,s}^{(\text{I})} = \nu_{r-1,s}\ ,\qquad\qquad
  \nu_{r,s}=\nu_{r,s}^{(\text{I})}\ ,
\end{displaymath}
namely
\begin{displaymath}
  \nu_{r,s}=\nu_{r-1,s}=\ldots=\nu_{s,s}\ .
\end{displaymath}

\smallskip
Let us introduce the quantities $b(\rmI;r,s,l)$ and $b(r,s,l)$ (being
$r$, $s$ and $l$ positive integers) that will be useful to
control the exponents of the $\Xi_r$ in the normalization
procedure,
$$
b(\rmI;r,s,l) =\left\{
\begin{aligned}
  &s &\qquad & \hbox{if } r=1\ ,\cr
  &0 &\qquad &\hbox{if } r\geq2, s=0\ ,\cr
  &\scriptstyle 3s-\left\lfloor\frac{s+r-1}{r}\right\rfloor-\left\lfloor\frac{s+r-2}{r}\right\rfloor-2 &\qquad &\hbox{if } r\geq2, 0<s\leq r, l = 0\cr
  &\scriptstyle 3s-\left\lfloor\frac{s+r-1}{r}\right\rfloor-\left\lfloor\frac{s+r-2}{r}\right\rfloor-1 &\qquad &\hbox{if } r\geq2, r<s\leq 2r, l = 0\cr
  &\scriptstyle 3s-\left\lfloor\frac{s+r-1}{r}\right\rfloor-\left\lfloor\frac{s+r-2}{r}\right\rfloor-1 &\qquad &\hbox{if } r\geq2, 0<s\leq r, l = 2\cr
  &\scriptstyle 3s-\left\lfloor\frac{s+r-1}{r}\right\rfloor-\left\lfloor\frac{s+r-2}{r}\right\rfloor &\qquad &\hbox{in the other cases} \cr
\end{aligned}
\right.
$$
and
$$
b(r,s,l) =\left\{
\begin{aligned}
  &0 &\qquad & \hbox{if } r>0, s=0\cr
  &\scriptstyle 3s-\left\lfloor\frac{s+r-1}{r}\right\rfloor-w_l &\qquad &\hbox{if } r=1, s>0\ ,\cr
  &\scriptstyle 3s-\left\lfloor\frac{s+r-1}{r}\right\rfloor-\left\lfloor\frac{s+r-2}{r}\right\rfloor-2 &\qquad &\hbox{if } r\geq2, 0<s\leq r, l = 0\cr
  &\scriptstyle 3s-\left\lfloor\frac{s+r-1}{r}\right\rfloor-\left\lfloor\frac{s+r-2}{r}\right\rfloor-1 &\qquad &\hbox{if } r\geq2, r<s\leq 2r, l = 0\cr
  &\scriptstyle 3s-\left\lfloor\frac{s+r-1}{r}\right\rfloor-\left\lfloor\frac{s+r-2}{r}\right\rfloor-1 &\qquad &\hbox{if } r\geq2, 0<s\leq r, l = 2\cr
  &\scriptstyle 3s-\left\lfloor\frac{s+r-1}{r}\right\rfloor-\left\lfloor\frac{s+r-2}{r}\right\rfloor &\qquad &\hbox{in the other cases} \cr
\end{aligned}
\right.
$$
with $w_0 = 2$, $w_2 = 1$ and $w_l=0$ for $l\geq2$.

\smallskip
We are now ready to state the main Lemma collecting the estimates for
the generic $r$-th normalization step of the normal form algorithm.

\begin{lemma} \label{lem:lemmone}
  Consider a Hamiltonian $H^{(r-1)}$ expanded as in~\frmref{frm:H(r-1)-f}.  Let
  $\chi^{(r)}_{0}=X^{(r)}_{0} + \prsca{\zeta^{(r)}}{\vphi}$ and $\chi_2^{(r)}$ be
  the generating functions used to put the Hamiltonian in normal form at order
  $r$, then one has
  \begin{equation}
  \begin{aligned}
    \|X_0^{(r)}\|_{1-d_{r-1}} &\leq \frac{1}{\omega} \nu_{r-1,r} \Xi_r^{3r-4}  E\epsilon^r\ ,\cr
    |\zeta^{(r)}| &\leq  \frac{1}{4 m \rho} \nu_{r-1,r} \Xi_r^{3r-3}  E\epsilon^r\ ,\cr
    \|\chi_2^{(r)}\|_{1-d_{r-1}-\delta_r} &\leq \frac{1}{\omega}
    3\nu_{r-1,r} \Xi_r^{3r-3}  \frac{E}{4}\epsilon^r\ .
    \label{frm:chi-estimates}
  \end{aligned}
  \end{equation}
  The terms appearing in the expansion of $H^{(\rmI;r-1)}$
  in~\frmref{frm:H(I;r-1)-f} are bounded as
  \begin{equation*}
  \|f_{l}^{(\rm I;r-1,s)}\|_{1-d_{r-1}-\delta_r}  \leq \nu^{(\rm I)}_{r,s}\Xi_r^{b(\rmI;r,s,l)}
   \frac{E}{2^l}\epsilon^s\ .
  \end{equation*}
  The terms appearing in the expansion of $H^{(r)}$
  in~\frmref{frm:f^(r,s)} are bounded as
  \begin{equation*}
    \|f_{l}^{(r,s)}\|_{1-d_r} \leq \nu_{r,s} \Xi_r^{b(r,s,l)} 
    \frac{E}{2^l}\epsilon^s\ .
  \end{equation*}
\end{lemma}
The proof of Lemma~\lemref{lem:lemmone} is deferred to Section~\ref{sbs:proof-lemmone}.

\subsection{Proof of Proposition~\ref{pro:forma-normale-1}}
We give here a sketch of the proof of
Proposition~\ref{pro:forma-normale-1}.  The proof is based on standard
arguments in the Lie series theory, that we recall here, referring to,
e.g.,~\cite{GioLoc-1997, GioLocSan-2014, SanCec-2017}, for more
details.

We give an estimate for the canonical transformation.  We denote
by $(\hat{p}^{(0)}, \hat{q}^{(0)})$ the original coordinates, and by
$(\hat{p}^{(r)}, \hat{q}^{(r)})$ the coordinates at step $r$. We also
denote by $\phi^{(r)}$ the canonical transformation mapping
$(\hat{p}^{(r)}, \hat{q}^{(r)})$ to $(\hat{p}^{(r-1)},
\hat{q}^{(r-1)})$, precisely
$$
\begin{aligned}
  \hat{p}^{(r-1)} &= \exp(L_{\chi_0^{(r)}}) \hat{p}^{(\rmI,r-1)} =
  \hat{p}^{(\rmI,r-1)} + \parder{\chi_0^{(r)}}{\hat{q}^{(r-1)}}\ ,\cr
  \hat{p}^{(\rmI,r-1)} &= \exp(L_{\chi_2^{(r)}}) \hat{p}^{(r)} =
  \hat{p}^{(r)} + \sum_{s\geq1}\frac{1}{s!} L_{\chi_2^{(r)}}^{s-1}
  \parder{\chi_2^{(r)}}{\hat{q}^{(r)}}\ ,\cr \hat{q}^{(r-1)} &=
  \exp(L_{\chi_2^{(r)}}) \hat{q}^{(r)} = \hat{q}^{(r)} -
  \sum_{s\geq1}\frac{1}{s!} L_{\chi_2^{(r)}}^{s-1}
  \parder{\chi_2^{(r)}}{\hat{p}^{(r)}}\ .
\end{aligned}
$$

Consider now a sequence of domains $\Dscr_{(3d-d_r)(\rho,\sigma)}$,
using Lemma~\lemref{lem:lemmone} we get
\begin{equation}
\begin{aligned}
  \Bigl|\hat{p}^{(r-1)} - \hat{p}^{(\rmI,r-1)}\Bigr| &<
  \left(\frac{1}{\omega e\delta_r \sigma}+\frac{1}{4 m \rho} \right)
  \Xi^{3r} \frac{100^r}{20} E \epsilon^r\ ,
  \cr
  \Bigl|\hat{p}^{(\rmI,r-1)} - \hat{p}^{(r)}\Bigr| &<
  \frac{1}{4 \omega e \delta_r \sigma} \Xi^{3r}
  \frac{100^r}{20} E \epsilon^r
  \sum_{s\geq1}
  \left(\frac{1}{\omega \delta_r^2 \rho \sigma} \Xi^{3r} \frac{100^{r}}{20} E \epsilon^r\right)^{s-1}\ ,
  \cr
  \Bigl|\hat{q}^{(r-1)} - \hat{q}^{(r)}\Bigr| &<
  \frac{1}{4 \omega \delta_r \rho} \Xi^{3r} \frac{100^r}{20} E
  \epsilon^r \sum_{s\geq1} \left(\frac{1}{\omega \delta_r^2 \rho
  \sigma} \Xi^{3r} \frac{100^{r}}{20} E \epsilon^r\right)^{s-1}\ .
\end{aligned}
\label{frm:trasf-coord}
\end{equation}

Thus if $\epsilon$ is small enough (for a very rough estimate take
$\epsilon < \frac{1}{100 \Xi^4}$) the series~\frmref{frm:trasf-coord}
defining the canonical transformation are absolutely convergent in the
domain $\Dscr_{(3d-d_{r-1}-\delta_r)(\rho,\sigma)}$, hence analytic.
Furthermore, one has the estimates
$$
|\hat{p}^{(r-1)} - \hat{p}^{(r)}| < \delta_r \rho\ ,\qquad
|\hat{q}^{(r-1)} - \hat{q}^{(r)}| < \delta_r \sigma\ .
$$
A similar argument applies to the inverse of $\phi^{(r)}$, which is defined as a composition of
Lie series generated by $−\chi_2^{(r)}$ and $-\chi_0^{(r)}$, thus we get
$$
\Dscr_{(3d-d_{r})(\rho,\sigma)}\subset \phi^{(r)}(\Dscr_{(3d-d_{r-1}-\delta_r)(\rho,\sigma)})\subset\Dscr_{(3d-d_{r-1})(\rho,\sigma)}\ .
$$
Consider now the sequence of transformations
$\Phi^{(\bar{r})}=\phi^{(1)}\circ\ldots\circ\phi^{(\bar{r})}$.  For
$(\hat{p}^{(r-1)}, \hat{q}^{(r-1)})\in\Dscr_{(3d-d_{r-1})(\rho,\sigma)}$ the transformation is clearly analytic and one has
$$
|\hat{p}^{(0)} - \hat{p}^{(\bar{r})}| <  \rho \sum_{j=1}^{\bar{r}} \delta_j\ ,\qquad
|\hat{q}^{(0)} - \hat{q}^{(\bar{r})}| <  \sigma \sum_{j=1}^{\bar{r}} \delta_j\ .
$$
Setting $d=\frac{1}{4}$ and using~\frmref{frm:delta_r}, one has $\sum_{j\geq1} \delta_j\leq
\frac{d}{2}=\frac{1}{8}$, thus~\frmref{frm:domini} immediately follows.
Finally, the estimates for the Hamiltonian in normal form had been
already gathered in Lemma~\lemref{lem:lemmone}.  This concludes the
proof of Proposition~\ref{pro:forma-normale-1}.


\section{Proof of Theorem~\ref{teo:forma-normale-r}}
\label{sec:teo}

In this Section we develop in a more detailed way the strategy used to
get Theorem~\ref{teo:forma-normale-r} from the normal form
constructed.  We have shown in the previous Section that, by mean of a
canonical and near to the identity change of coordinates, it is
possible to bring the original Hamiltonian the form
\eqref{frm:H(r)}. We have already stressed in the Introduction the
main feature of our construction: if one considers the approximate
equations of motion corresponding to the normal form truncated at
order $\Oscr\tond{\eps^r}$, when evaluated on $({q}={q}^*,\hat{p}=0)$,
they provide a periodic orbit of frequency $\omega$ once ${q}^*$
fulfills the already mentioned equation
\eqref{frm:qstar}. Generically, for $r\geq 2$, the value ${q}^*$ would
depend continuously on $\varepsilon$, precisely
${q}^*(\varepsilon)={q}_0^*+ \mathcal{O}(\varepsilon)$, with ${q}_0^*$
solution of the $\eps$-independent equation \eqref{frm:qstar.Poi}.

The periodicity of an orbit for the full Hamiltonian \eqref{frm:H(r)}
is given by
$$
\begin{aligned}
\hat{q}(T)-\hat{q}(0)-{\Lambda}T
&= \int_0^T \nabla_p\quadr{f_{4}^{(r,0)} + \sum_{s=1}^r f_2^{(r,s)}}\, ds +
\mathcal{O}(\vert {p}\vert^2) + \mathcal{O}(\varepsilon\vert
        {p}\vert) + \Oscr\tond{\varepsilon^{r+1}} = 0\ , \\
p_1(T)-p_1(0) &= \mathcal{O}(\varepsilon\vert {p}\vert^2) +
\Oscr\tond{ \varepsilon^{r+1}}=0\ ,\\ {p}(T)-{p}(0) &= -\int_0^T
\sum_{s=1}^r\nabla_q \quadr{f_0^{(r,s)}+f_2^{(r,s)}}\, ds +
\mathcal{O}(\varepsilon\vert {p}\vert^2) +
\Oscr\tond{\varepsilon^{r+1}}=0 \ ,
\end{aligned}
$$ where the unknown is the initial datum $(\hat q=\hat q(0),\hat
p=\hat p(0))$, namely the Cauchy problem. Due to the conservation of
the energy, we can eliminate the equation for $p_1$, divide the $n-1$
actions $p$ by $\eps$ and look at $q_{1}(0)$ as a parameter (the phase
along the orbit). The system of $2n-1$ equations in $2n-1$ unknowns
$(q(0),p_1(0), p(0))$
$$
\begin{aligned}
\hat{q}(T)-\hat{q}(0)-{\Lambda}T &=
\int_0^T \nabla_p\quadr{f_{4}^{(r,0)} + \sum_{s=1}^r f_2^{(r,s)}}\, ds
+ \mathcal{O}(\vert {p}\vert^2) + \mathcal{O}(\varepsilon\vert
{p}\vert) + \Oscr\tond{\varepsilon^{r+1}} = 0\ ,\\
\frac{{p}(T)-{p}(0)}{\eps} &= -\frac1\eps\int_0^T \sum_{s=1}^r\nabla_q
\quadr{f_0^{(r,s)}+f_2^{(r,s)}}\, ds + \mathcal{O}(\vert {p}\vert^2) +
\Oscr\tond{\varepsilon^{r}}=0\ ,
\end{aligned}
$$ takes the form \eqref{frm:Ups}. The approximate periodic solution
\begin{displaymath}
\hat{p}(t)=0\ ,\qquad q_1(t)=\omega t + q_{1}(0)\ ,\qquad {q}(t)={q}^*\ ,
\end{displaymath}
corresponds to (and actually represents) an approximate zero
$(q(0)=q^*,\hat p(0)=0)$ for the $\Upsilon$ map. The proof of
Theorem~\ref{teo:forma-normale-r} then simply consists in the
application of
\begin{proposition}[Newton-Kantorovich method]
\label{prop:N-K}
  Consider
$F\in\mathcal{C}^1\left(\mathcal{U}(x_0)\times\mathcal{U}(0),V\right)$.
Assume that there exist three constants $C_{1,2,3}>0$
dependent, for $\eps$ small enough, on $\Uscr(x_0)$ only, and two
parameters $0\leq 2\alpha<\beta$ such that
\begin{equation}
\label{frm:Newton-hp}
  \begin{aligned}
    \|F(x_0,\eps)\| &\leq C_1 |\eps|^{\beta}\ ,\\
    \|[F'(x_0,\eps)]^{-1}\|_{\Lscr(V)} &\leq C_2 |\eps|^{-\alpha}\ ,\\
    \norm{F'(z,\eps)-F'(x_0,\eps)}_{\Lscr(V)} &\leq C_3 \norm{z-x_0}\ .
  \end{aligned}
\end{equation}
Then there exist positive $C_0$ and $\eps^*$ such that, for
$|\eps|<\eps^*$, there exists a unique $x^*(\eps)\in\mathcal{U}(x_0)$
which fulfills
\begin{displaymath}
F(x^*,\eps)=0\ ,\qquad\qquad \|x^*-x_0\| \leq C_0|\eps|^{\beta-\alpha}\ .
\end{displaymath}
Furthermore, Newton's algorithm converges to $x^*$.
\end{proposition}
The proof of the Proposition is reported in Appendix~\ref{app:fixedpoint}.

Since we are seeking for a true periodic solution close to the
approximate one, we take $({q},\hat{p})$ in a small ball centered in
$({q^*},0)$; thus both the variables can be interpreted ``locally'' as
cartesian variables in $\RR^{2n-1}$. We have already introduced in
\eqref{frm:M(epsilon)} $M(\epsilon)$, the differential of the map
$\Upsilon$ evaluated in $({q},\hat{p})=({q}^*,0)$.  Expanding
$M(\epsilon)$ in powers of $\epsilon$ we get
$$ M(\eps) = M_0+\varepsilon M_1+\mathcal{O}(\varepsilon^2)=
\begin{pmatrix}
\varepsilon A_1+\mathcal{O}(\varepsilon^2) & C_0+\varepsilon
C_1+\mathcal{O}(\varepsilon^2) \\ B_0+\varepsilon
B_1+\mathcal{O}(\varepsilon^2) & D_0+\eps D_1+\mathcal{O}(\varepsilon^2) \\
\end{pmatrix}\ ,
$$
where
\begin{equation}
  \label{frm:B0-C0}
 B_{0;i,j}= -\quadr{\frac{\partial^2 f_0^{(r,1)}}{\partial {q_i}\partial
   {q_j}}\Big|_{q={q}_0^*}}\frac{T}{\eps}\ ,\qquad\qquad C_0 = C T\ ,
\end{equation}
and $C$ is the twist matrix defined in \eqref{frm:twist}. The first of
\eqref{frm:Newton-hp} is satisfied with $\beta=r$. The third of
\eqref{frm:Newton-hp} is satisfied because of the smoothness of the
flow at time $T$ w.r.t. the initial datum (it keeps the same
smoothness as its vector field). The core of the statement is then the
requirement on the invertibility of $M(\eps)$. If $B_0$ is invertible,
then the same holds for $M_0$ (being the twist $C_0$ invertible) which
is the leading order of $M$; hence $M(\eps)$ is also invertible and
the second of \eqref{frm:Newton-hp} is satisfied with $\alpha=0$,
being $M_0$ independent of $\eps$. This is actually Poincar\'e's
theorem. If instead $B_0$ has a non trivial Kernel, then the same
holds also for $M_0$, typically with a greater dimension. The required
invertibility of $M(\eps)$, asked by Theorem~\ref{teo:forma-normale-r}, is necessarily due to the
$\eps$-corrections, who are responsible for the bifurcations of the
zero eigenvalues of the matrix $M_0$. Hence, in order to fulfill the
second of \eqref{frm:Newton-hp}, we need the smallest eigenvalues of
$M(\eps)$ to bifurcate from zeros as $\lambda_j(\eps)\sim
\eps^\alpha$, with $\alpha<\frac{r}2$, which is indeed
\eqref{frm:stima-autovalori}. Finally, estimates \eqref{frm:estim.sol}
are of the same type as the one in Proposition~\ref{prop:N-K}, even
after back-transforming the solutions to the original canonical
variables with $\Phi^{(r)}$. Indeed, as illustrated in the detailed
proof of Proposition~\ref{pro:forma-normale-1}, the normalizing
transformation $\Phi^{(r)}$ is a near the identity transformation.


\section{One parameter families.}
\label{sec:applic}

Generically we expect that, apart from very pathological examples, two
normal form steps are enough to get a clear insight into the
degeneracy. In particular, with a second order approximation one can
investigate whether one-parameter families ${q}^*_0(s)$, which are
solutions of \eqref{frm:qstar.Poi}, are or not destroyed. In the first
case, the isolated solutions which survive to the breaking of the
family are natural candidates for the continuation, once
\eqref{frm:stima-autovalori} has been verified. In the second case, at
least a third step of normalization is necessary, unless there are
good reasons to believe that the whole family survives, due to the
effect of some hidden symmetry of the model.

What we are going to develop in the first part of this Section is
exactly the case when the first of \eqref{frm:B0-C0} admits
one-parameter families of solutions on the torus $\TT^{n-1}$, which
means that $\dim\tond{\Ker(B_0)}=1$. In this easier case (which
represents the weakest degeneracy for $B_0$), under suitable
conditions on the matrix $M_0$, it is possible to apply some results
of perturbation theory of matrices to $M(\eps)$ (see \cite{YakS75},
Chap. IV, par. 1.4) {in order to replace assumption
  \eqref{frm:stima-autovalori} with a more accessible
  criterion.} This allows to get a more applicable formulation of
Theorem~\ref{teo:forma-normale-r}, which will be used in the
forthcoming application.

\subsection{Some few facts on matrix perturbation theory}

The degeneration we are here considering implies that
$0\in \Spec(B_0)$, with the geometric multiplicity being equal
to one ($m_g(0,B_0)=1$). Let ${a}_1$ be the $(n-1)$-dimensional vector
generating $\Ker(B_0)$. Let us introduce also ${f_1}$ as the embedding
of ${a}_1$ into $\RR^{2n-1}$, namely the $(2n-1)$ vector
$$
{f_1}=\begin{pmatrix} {a}_1 \\ {0}
\end{pmatrix} \ .
$$
We have the following

{\begin{lemma}
\label{lem:M0.spec}
Assume that the kernel of $M_0$ is of dimension one and is generated by ${f_1}$, namely
$\Ker(M_0)=\Span({f_1})\,$. If the
following orthogonality condition is fulfilled
\begin{equation}
\label{frm:orto.cond}
\Bigl\langle{C_0^{-1}D_0^\top{a_1}, \begin{pmatrix} {a}_1 \\ {0}
\end{pmatrix}}\Bigr\rangle=0\ ,
\end{equation}
then the algebraic multiplicity of the zero eigenvalue is greater
than two ($m_a(0,M_0)\geq 2$).
\end{lemma}
}

\proof
In order to study the $\Ker(M_0)$, we have to solve
$$ 
\begin{pmatrix}
O & C_0 \\ B_0 & D_0 \\
\end{pmatrix} 
\begin{pmatrix}
{x} \\- {y} \\
\end{pmatrix} 
=
\begin{pmatrix}
C_0{y} \\ B_0{x} -D_0{y}\\
\end{pmatrix} 
=
\begin{pmatrix}
{0} \\ {0} \\
\end{pmatrix} 
$$ which gives, due to the invertibility of $C_0$, ${y}={0}$, and thus
${x}\in\text{Ker}(B_0)$.  This provides the first claim. The statement
concerning the algebraic multiplicity can be derived investigating the
Kernel of the adjoint matrix $M_0^\top$. It is easy to see that
$$ \text{Ker}(M_0^\top)= \Span\tond{g} \qquad\qquad {g}=
\begin{pmatrix}-C_0^{-1}D_0^\top{a_1} \\ {a}_1 \\
\end{pmatrix}
$$ {and deduce that the assumption \eqref{frm:orto.cond} is equivalent
  to $\inter{{f_1},{g}}=0$, where the right hand vector in
  \eqref{frm:orto.cond} is the $n$-dimensional vector built by
  complementing ${a_1}$ with one $0$}. The last, according to Lemma
III, Chapter 1.16 of \cite{YakS75}, is not compatible with
$m_a(0,M_0)=1$. Precisely, we can observe that the orthogonality
condition between the two vectors allows to find a second generalized
eigenvector ${f_2}$ for $\Ker(M_0)$, as a solution of $M_0
{f_2}={f_1}$. Indeed, the Fredholm alternative theorem guarantees the
existence of ${f_2}$ under exactly the condition
$\inter{{f_1},{g}}=0$. \qed

In order to determine the asymptotic behaviour of the eigenvalues
$\lambda(\eps)\in\text{spec}(M(\eps))$, we make us of the fact that
$dim(\text{Ker}(M_0))=1$ and that the following Lemma holds (refer
to \cite{YakS75}, Cap.IV, \S 1. for all the details)

\begin{lemma} 
\label{lem:stima-biforcazione}
Let $\lambda_0$ an eigenvalue $M_0$ with $m_g(\lambda_0,M_0)=1$ and
$m_a(\lambda_0,M_0)=h\geq 2$ and let ${f}_1,\ldots,{f}_h$ the
generalized eigenvectors relative to $\lambda_0$, defined by the
recursive scheme
$$ M_0 {f}_1=\lambda_0 {f}_1, \quad M_0 {f}_2=\lambda_0
{f}_2 +{f}_1, \ldots, M_0 {f}_h=\lambda_0 {f}_h
+{f}_{h-1}.
$$ Moreover, let ${g_1},\ldots,{g_h}$ the generalized
eigenvectors for $M_0^\top$ relative to
$\lambda_0$, such that
$$ \langle {f}_j, {g}_i \rangle=\delta_{ji}, \quad \text{con}
\quad j,i=1,\ldots,h
$$ and define 
\begin{equation*}
\gamma=\langle M_1 {f}_1,{g}_h \rangle\ .
\end{equation*}
If $\gamma\neq
0$, then the $h$ solutions $\lambda_j(\varepsilon)$ of the characteristic equation
$$ \det(M(\eps)-\lambda I)=0
$$ are given by
$$ \lambda_j(\varepsilon)=\lambda_0-(\varepsilon \gamma)^{1/h}_j
+\mathcal{O}(\varepsilon^{2/h})\ ,
$$ where $(\varepsilon \gamma)^{1/h}_j$ are the $h$ distinct roots of
$\sqrt[h]{\varepsilon\gamma}$.
\end{lemma}

\subsection{The special case of $m_a(0,M_0)=2$.}

We are interested in the bifurcations of the zero
eigenvalue (needed to bound the inverse matrix $M^{-1}(\eps)$), thus
in the previous Lamma \ref{lem:stima-biforcazione} we can take
$\lambda_0=0$ and ${f}_1$ as the eigenvector generating
$\text{Ker}(M_0)$.  Moreover, since
$$
\begin{pmatrix}
A_1 & C_1 \\ B_1 & D_1 \\
\end{pmatrix} 
\begin{pmatrix}
{a}_1 \\ {0} \\
\end{pmatrix} 
=
\begin{pmatrix}
A_1{a}_1 \\ B_1{a}_1 \\
\end{pmatrix} 
$$ the value of $\gamma$ does not depend on the whole matrix $M_1$,
but only on the blocks $A_1$ and $B_1$. The problem is further
simplified when $m_a(0,M_0)=2$: in this case ${g}_2$ coincides with
${g}$ and $\gamma$ reduces to
$$ \gamma=\langle M_1{f}_1,{g}_2 \rangle= 
\Bigl\langle{\begin{pmatrix}
A_1{a}_1 & B_1{a}_1 \\
\end{pmatrix} 
, \begin{pmatrix} -C_0^{-1} D_0^\top {a_1} \\ {a}_1 \\
\end{pmatrix}}\Bigr\rangle 
= \Bigl\langle \tond{B_1 - D_0 C_0^{-1}A_1}{a}_1,{a}_1 \Bigr\rangle\ .
$$ Thus, under the easier condition
\begin{displaymath}
  \gamma = \langle \tond{B_1 - D_0 C_0^{-1}A_1}{a}_1,{a}_1 \rangle\not=0\ .
\end{displaymath}

Theorem~\ref{teo:forma-normale-r} can be formulated as

\begin{theorem}
  \label{teo:forma-normale-2}
  Consider $\Upsilon=\tond{{F},{G}}$ defined by \eqref{frm:Ups}
  in a neighbourhood of $({0},{q^*})$, with ${q^*}(\eps)$ defined by
  \eqref{frm:qstar} and $r=2$. Let $\dim(\Ker(B_0))=1$, being ${a}_1$
  its generator. Assume also that $m_a(0,M_0)=2$ and that it holds
  \begin{equation}
    \label{frm:stima-autovalori-bis}
  \inter{\tond{B_1 - D_0 C_0^{-1}A_1}{a}_1,{a}_1}\not=0\ .
  \end{equation}
  Then, there exist positive constants $C_0$ and $\eps^*$ such that, for 
  $|\eps|<\eps^*$ there exists a point
  $({q}_{\rm p.o.}(\eps),\hat{p}_{\rm p.o.}(\eps))\in\Uscr\times\TT^{n-1}$ which
  solves
  \begin{displaymath}
    \Upsilon({q}_{\rm p.o.},\hat{p}_{\rm p.o.};\eps,q_1(0))= 0\ ,\qquad\qquad
    \norm{({q}_{\rm p.o.},\hat{p}_{\rm p.o.})-({q^*},0)}\leq C_0\eps^{3/2}\ .
  \end{displaymath}
\end{theorem}

In order to verify condition~\eqref{frm:stima-autovalori-bis}, the
block matrices $A_1$ and $B_1$ are needed; as a consequence, the first
order corrections to the generic Cauchy problem, $\hat{q}^{(1)}(t)$
and $\hat{p}^{(1)}(t)$ have to be derived. With a standard approach,
as the one performed in \cite{MelS05}, and after expanding in $\eps$
both the period map $\Upsilon$ and the solution ${q^*}(\eps)=q_0 +
\Oscr(\eps)$ one gets {\begin{equation}
  \begin{aligned} 
A_1 &= -\frac{T^2}2 C_0 D_q\nabla_{\hat
  q}f^{(2,1)}_0({q^*_0})+TD_q\nabla_{\hat p} f^{(2,1)}_2({q^*_0})
\\ B_1 &= -T D^3_{{q}} f^{(2,1)}_0({q^*_0}){q^*_1}-T
D^2_{{q}} f^{(2,2)}_0({q^*_0}) \\ &\qquad+
\frac{T^2}2\left[D^2_{{qp}} f^{(2,1)}_2({q^*_0})
  D^2_{{q}} f^{(2,1)}_0({q^*_0}) - D^2_{{q}} f^{(2,1)}_0({q^*_0})
  D^2_{{qp}} f^{(2,1)}_2({q^*_0})\right] \nonumber\\ &\qquad+ \frac{T^3}6
\left[D^2_{{q}} f^{(2,1)}_0({q^*_0})\right]^2\nonumber\ .
\end{aligned}
\end{equation}
Despite the formulation of Theorem~\ref{teo:forma-normale-2} is
simplified with respect to the abstract result stated in Theorem~\ref{teo:forma-normale-r}, it is evident from the above formulas that
it can be a hard task to verify condition
\eqref{frm:stima-autovalori-bis}. However, if the original Hamiltonian
is even in the angle variables, as often happens in models of weakly
interacting anharmonic oscillators, the condition
\eqref{frm:stima-autovalori-bis} can be further simplified if the
solutions to be investigated are the in/out-of-phase solutions
$q^*=0,\pi$, as shown in the following example.}

\subsection{Example: square dNLS cell with nearest neighbour interaction}

Let us consider the Hamiltonian system in real coordinates
$$ H = H_{0}+\varepsilon H_{1} = \sum_{j=1}^{4}\Biggl( \frac{x_j^2 +
  y_j^2}{2} + \Biggl( \frac{x_j^2 + y_j^2}{2}\Biggr)^2 +\varepsilon
(x_{j+1}x_j+ y_{j+1}y_j)\Biggr)\ ,
$$ which, introducing the action-angle variables $(x_j,y_j) =
(\sqrt{2I_j}\cos\varphi_j,\sqrt{2I_j}\sin\varphi_j)$, reads
$$ H=\sum_{j=1}^{4}\left(I_j+I_j^2+
2\varepsilon\sqrt{I_{j+1}I_j}\cos(\varphi_{j+1}-\varphi_j)\right)\ .
$$ Let us now fix the fully resonant torus
${I}^*=(I^*,I^*,I^*,I^*)$ and make a Taylor expansion around
${I}^*$.  The unperturbed part, $H_0$, reads
\begin{displaymath}
  H_0({I}) = 4I^* +4(I^*)^2 + (1+2I^*)(J_1+ J_2+J_3+J_4) +
  J_1^2+J_2^2+J_3^2+J_4^2\ ,
\end{displaymath}
while the perturbation $H_1$ takes the form
$$
\begin{aligned}
H_1({I},{\varphi}) &=  2I^*(\cos(\varphi_2
-\varphi_1) + \cos(\varphi_3 -\varphi_2) +\cos(\varphi_4
-\varphi_3)+\cos(\varphi_4 -\varphi_1)) \\ & \qquad +
(J_1+J_2)\cos(\varphi_2 -\varphi_1) + (J_3+J_2)\cos(\varphi_3
-\varphi_2) \\ & \qquad + (J_4+J_3)\cos(\varphi_4 -\varphi_3) +
(J_1+J_4)\cos(\varphi_4 -\varphi_1)  + \mathcal{O}(\vert
{J}\vert^2)\ .
\end{aligned}
$$
We introduce\footnote{In this case, we have preferred the angles
    to be the relative phase differences among consecutives angles,
    rather than the phase differences with respect to the first angle
    $\vphi_1$.}  the resonant angles $\hat q=(q_1,q)$ and
their conjugate actions $\hat p=(p_1,p)$
$$ \left\lbrace \begin{aligned} & q_1 = \varphi_1 \\ & q_2 = \varphi_2
  - \varphi_1 \\ & q_3 = \varphi_3 - \varphi_2 \\ & q_4 = \varphi_4 -
  \varphi_3
\end{aligned}
\right. , \qquad \left\lbrace \begin{aligned} & p_1 = J_1 + J_2 + J_3
  + J_4 \\ & p_2 = J_2 + J_3 + J_4\\ & p_3= J_3 + J_4 \\ & p_4 = J_4
\end{aligned}
\right. \ .
$$
Thus, ignoring the constant terms, we can rewrite $H$ as
$$
\begin{aligned}
H &= \omega p_1 + \Bigl((p_1 -p_2)^2 + (p_2 -p_3)^2 + (p_3 -p_4)^2
+p_4^2\Bigr)+ \\ &\qquad\quad+ \varepsilon \Bigl(2I^* \cos(q_2)+ 2I^* \cos(q_3)+
  2I^* \cos(q_4)+ 2I^* \cos(q_2+q_3+q_4)\Bigr) \\ &\qquad\quad +
  (p_1-p_3)\cos(q_2) + (p_2-p_4)\cos(q_3) +p_3\cos(q_4) \\ & \qquad
  \quad + (p_1 -p_2 +p_4)\cos(q_2+q_3+q_4)+
\mathcal{O}(\varepsilon\vert \hat{p}\vert^2) \\ & = \omega p_1 +
f_4^{(0,0)}(p_1,p_2,p_3,p_4) + f_0^{(0,1)}(q_2,q_3,q_4) \\ &\qquad\quad +
f_2^{(0,1)}(p_1,p_2,p_3,p_4,q_2,q_3,q_4) +
\mathcal{O}(\varepsilon\vert \hat{p}\vert^2)\ ,
\end{aligned} 
$$ where $\omega=1+2I^*$.

\begin{remark}
With the usual canonical complex coordinates $\psi_j =
\frac1{\sqrt2}\tond{x_j+\im y_j}$, the Hamiltonian reveals to be a
dNLS model, with periodic boundary conditions
\begin{equation}
  \label{frm:H-dNLS}
H = \sum_{j=1}^4\quadr{|\psi_j|^2 + |\psi_j|^4 +
  \eps\tond{\psi_{j+1}\overline\psi_j+c.c.}}\ ,\qquad \psi_0=\psi_4\ .
\end{equation}
In agreement with this, we observe that the Hamiltonian does not
depend on the fast angle $q_1$. This is due to the effect of the Gauge
symmetry of the model, as visible in the complex form
\eqref{frm:H-dNLS}.  As a consequence, $f_0^{(0,1)}(q_2,q_3,q_4)$ is
already in normal form and the first stage only consists in the
translation of the actions, which allows to keep fixed $\omega$.
\end{remark}

Since $f_2^{(0,1)}$ is automatically averaged w.r.t. $q_1$, the
homological equation defining $\zeta^{(1)}$ is equivalent to the
following linear system
$$ \inter{{\nabla_{\hat{p}} f_4^{(0,0)}}, 
{\zeta}^{(1)}} + f_{2}^{(0,1)} \Bigr|_{{q}={q}^*}= 0\ ,
$$ whose solution is given by
$$
\left\lbrace \begin{aligned} & \zeta_1^{(1)}=-\quadr{\cos(q_2^*) +
  \cos(q_3^*) + \cos(q_4^*) + \cos(q_2^* + q_3^* + q_4^*)} \\ &
  \zeta_2^{(1)}=-\quadr{\frac{\cos(q_2^*)}{2} + \cos(q_3^*) + \cos(q_4^*) +
  \frac{\cos(q_2^* + q_3^* + q_4^*)}{2}} \\ & \zeta_3^{(1)}=
  -\quadr{\frac{\cos(q_3^*)}{2} + \cos(q_4^*) +
  \frac{\cos(q_2^*+q_3^*+q_4^*)}{2}} \\ & \zeta_4^{(1)}=-\quadr{
  \frac{\cos(q_4^*)}{2} + \frac{\cos(q_2^*+q_3^*+q_4^*)}{2}}
\end{aligned} \right. \ .
$$
Since the normal form preserves the symmetry, the newly generated
term $f_2^{(\text{I};0,1)}$ is again independent of $q_1$ and no
further average is required. The values ${q}^*$, which define the
approximate periodic orbit at leading order, are given by the
solutions of the trigonometric system (depending only on sines, due to
the parity of the Hamiltonian)
\begin{equation*}
  \left\lbrace \begin{aligned} & -2I^*\sin(q_2)-2I^*\sin(q_2+q_3+q_4)
  = 0 \\ & -2I^*\sin(q_3)-2I^*\sin(q_2+q_3+q_4) = 0 \\ &
  -2I^*\sin(q_4)-2I^*\sin(q_2+q_3+q_4) = 0
\end{aligned}
\right. \ .
\end{equation*}
Such solutions are given by the two isolated configurations $(0,0,0)$,
$(\pi,\pi,\pi)$, and the three one-parameter families
$Q_1=(\vartheta,\vartheta,\pi-\vartheta)$,
$Q_2=(\vartheta,\pi-\vartheta,\vartheta)$,
$Q_3=(\vartheta,\pi-\vartheta,\pi-\vartheta)$, with $\theta\in S^1$,
which all intersect in the two opposite configurations
$\pm(\frac{\pi}{2},\frac{\pi}{2},\frac{\pi}{2})$.  Since the twist
condition \eqref{frm:twist} is verified, we only need \eqref{frm:qstar.Poi}
in order to apply the implicit function theorem (which reduces to the
classical result of Poincar\'e).  Factoring out $-2I^*$, the
non-degeneracy condition reads
$$ \left| \begin{pmatrix} \cos(q_2^*)+\cos(q_2^*
  +q_3^*+q_4^*) & \cos(q_2^* +q_3^*+q_4^*) & \cos(q_2^* +q_3^*+q_4^*)
  \\ \cos(q_2^* +q_3^*+q_4^*) & \cos(q_3^*)+\cos(q_2^* +q_3^*+q_4^*) &
  \cos(q_2^* +q_3^*+q_4^*) \\ \cos(q_2^* +q_3^*+q_4^*) & \cos(q_2^*
  +q_3^*+q_4^*) & \cos(q_4^*)+\cos(q_2^* +q_3^*+q_4^*) \\
\end{pmatrix}\right|\neq 0.
$$ If we evaluate the determinant in the two isolated configurations,
we get $\det(B_0)=\pm 4T\neq 0$, hence the corresponding solutions can
be continued for small enough $\eps$. In the three families we
obviously get a degeneration, since the tangent direction to each
family represents a Kernel direction, hence $\det
\bigl({B_0\big|_{Q_j}}\bigr)=0$. Furthermore in the intersections
$\pm(\frac{\pi}{2},\frac{\pi}{2},\frac{\pi}{2})$ the matrices are
identically zeros. For all these families a second normalization step
is thus needed.

The first stage of the second normalization step deals with
$$ f_0^{(1,2)}= f^{(\text{I};0,2)}_{0} =
L_{\langle{\zeta}^{(1)},\hat{q}\rangle} f^{(0,1)}_{2} +
\frac{1}{2}L_{\langle{\zeta}^{(1)},\hat{q}\rangle}^2
f^{(0,0)}_{4}\ ,
$$ which is already averaged over $q_1$, due to the preservation of
the symmetry. The same holds also for the linear term in the action
variables $f_2^{(1,2)}$, given by
$$f_2^{(1,2)}= f^{(\text{I};0,2)}_{2} = 
L_{\langle{\zeta}^{(1)},\hat{q}\rangle} f^{(0,1)}_{4}\ .
$$
Hence, the homological equation providing the new translation
$\zeta^{(2)}$ reads
$$ L_{\langle{\zeta}^{(2)},\hat{q}\rangle} f_{4}^{(0,0)} +
L_{\langle{\zeta}^{(1)},\hat{q}\rangle} f_{4}^{(0,1)}
\Bigr|_{{q}={q}^*} = 0\ .
$$
The new linear term in the action
$$
f_{2}^{(\text{I};1,2)} = 
L_{\langle{\zeta}^{(1)},\hat{q}\rangle} f_{4}^{(0,1)} +
L_{\langle{\zeta}^{(2)},\hat{q}\rangle} f_{4}^{(0,0)}\ ,
$$ is again already averaged over $q_1$, hence the second step is
concluded, and the transformed Hamiltonian reads
$$
\begin{aligned}
H^{(2)} &= \omega p_1 + f_4^{(0,2)}(\hat{p})
\\ & \quad + {f}_0^{(1,2)}({q}) +
{f}_2^{(1,2)}(\hat{p},{q}) \\ & \quad
+{f}_0^{(2,2)}({q}) +
{f}_2^{(2,2)}(\hat{p},{q}) \\ & \quad
+\Oscr(\varepsilon\vert \hat{p}\vert^2) +
\Oscr\tond{\eps^3}\ .
\end{aligned}
$$ The approximate periodic orbit corresponds to the
${q}^*$ for which
\begin{equation*}
\nabla_q \quadr{{f}_0^{(1,2)}({q}) +{f}_0^{(2,2)}({q})} = \nabla_q
      {f}_0^{(1,2)}({q}) + \nabla_q\Bigl\langle{\nabla_{\hat{p}}f_2^{(0,1)}(q),\,
        {\zeta}^{(1)}}\Bigr\rangle = 0\ ,
\end{equation*}
where in the correction due to $f_0^{(2,2)}$, only the term
$L_{\langle{\zeta}^{(1)},\hat{q}\rangle} f^{(0,1)}_{2}$ really
matters, having a non trivial dependence on the slow angles $q$. By
exploiting the explicit expression for ${\zeta_1}$ previously
derived,
and replacing ${q}^*$ with ${q}$ in it, we explicitly get the system
$$ \left\lbrace \begin{aligned} -8\left( \sin(q_2)+\sin(q_2+q_3+q_4)
  \right) &+ \varepsilon \bigg[ 2\sin(2q_2) + \sin(q_2-q_3)
    +2\sin(q_2+q_3) \\ & \qquad + 2\sin(2q_2+ 2q_3+ 2q_4) +2\sin(2q_2+
    q_3+q_4) \\ & \qquad +\sin(q_2+ q_3+2q_4) \bigg]=0 \\ -8\left(
  \sin(q_3)+\sin(q_2+q_3+q_4) \right) &+ \varepsilon \bigg[
    2\sin(2q_3) + \sin(q_3-q_2) +2\sin(q_2+q_3) \\ & \qquad
    +\sin(q_3-q_4) + 2\sin(q_3+q_4) \\ & \qquad + 2\sin(2q_2+ 2q_3+
    2q_4) +\sin(2q_2+ q_3+q_4) \\ & \qquad +\sin(q_2+ q_3+2q_4)
    \bigg]=0 \\ -8\left( \sin(q_4)+\sin(q_2+q_3+q_4) \right) &+
  \varepsilon \bigg[ 2\sin(2q_4) + \sin(q_4-q_3) +2\sin(q_3+q_4) \\ &
    \qquad + 2\sin(2q_2+ 2q_3+ 2q_4) +\sin(2q_2+ q_3+q_4) \\ & \qquad
    +2\sin(q_2+ q_3+2q_4) \bigg]=0 \\
\end{aligned}
\right. \ ,
$$ depending on the effective small parameter
$\tilde{\varepsilon}=\frac{\varepsilon}{I^*}$.  The above system has
the  structure
\begin{equation}
  \label{frm:F(q,epsilon)}
  F(q,\varepsilon)= F_0 (q)+ \varepsilon F_1 (q)=0\ ,
\end{equation}
where $F:\mathbb{T}^3 \times \mathcal{U}(0) \rightarrow
\mathbb{R}^3$. Moreover, we have already found at first normalization
step that
\begin{equation*}
F(Q_j(\theta),0) = F_0(Q_j(\theta)) = 0\ .
\end{equation*}
Suppose that there exists a solution $q(\eps) =
(q_2(\varepsilon),q_3(\varepsilon),q_4(\varepsilon))$ which is at
least continuous in the small parameter, i.e. $\mathcal{C}^0
(\mathcal{U}(0),\mathbb{T}^3)$.  Hence, from continuity, we must have
$$ \lim_{\varepsilon\rightarrow 0}
F(q_2(\varepsilon),q_3(\varepsilon),q_4(\varepsilon),\varepsilon) =
F_0 (q_2(0),q_3(0),q_4(0))=0\ ,
$$ which means that $q(0)\in Q_j$. Let us introduce the matrices
$\tilde{B}_{0,j}(\vartheta)=\frac{\partial
  F_0(Q_j(\vartheta))}{\partial {q}}$ and observe that the tangent
directions to the three families
$$ \partial_{\vartheta}Q_1=\begin{pmatrix} 1 \\ 1 \\ -1
\end{pmatrix},
\qquad \partial_{\vartheta}Q_2=\begin{pmatrix} 1 \\ -1 \\ 1
\end{pmatrix}
\qquad \text{and} \qquad \partial_{\vartheta}Q_3=\begin{pmatrix} 1 \\ -1
\\ -1
\end{pmatrix}
$$ represent the Kernel direction of $\tilde{B}_{0,j}$, 
for $j=1,2,3$, respectively. A standard proposition of bifurcation theory provides a
necessary condition for the existence of a solution
$Q_j(\theta,\eps)$ which is a continuation of $Q_j(\theta)$

\begin{proposition} \label{prop:CN-continuazione-famiglie}
Necessary condition for the existence of a solution
$q(\eps) = Q_j(\vartheta,\varepsilon)$ of \eqref{frm:F(q,epsilon)} is that
$$ F_1(Q_j(\vartheta,0)) \in
\Range(\tilde{B}_{0,j}(\vartheta))\ .
$$ If $\tilde{B}_{0,j}(\vartheta)$ is symmetric, the above condition
simplifies
\begin{equation} \label{frm:kernel-hp}
F_1(Q_j(\vartheta,0)) \perp \Ker(\tilde{B}_{0,j}(\vartheta)).
\end{equation}
\end{proposition}   
Let us apply the above Proposition to show that the families $Q_1$ and
$Q_3$ break down.  Precisely, all their points, except for
those corresponding to $\theta=\{0,\pi/2,\pi\}$, do no represent true
candidates for the continuation.  We compute
$\inter{F_1(Q_j(\vartheta,0)),\partial_\theta Q_j}$ for $j=1,3$
\begin{displaymath}
\langle F_1(Q_1(\vartheta)), \partial_{\vartheta}Q_1 \rangle = 8
  \sin(2\vartheta) = \langle F_1(Q_3(\vartheta)),
  \partial_{\vartheta}Q_3 \rangle\ ,
  \end{displaymath}
which shows that the necessary condition is generically violated
for the two families $Q_{1,3}$, apart from the in/out-of-phase
configurations $(0,0,\pi)$, $(\pi,\pi,0)$, $(0,\pi,\pi)$, $(\pi,0,0)$
and the symmetric vortex configurations $\pm\left(
\frac{\pi}{2},\frac{\pi}{2},\frac{\pi}{2} \right)$, the last being
also points of $Q_2(\theta)$.

A way to conclude that the above mentioned in/out-of-phase
configurations can be continued to periodic solutions is to apply
Theorem~\ref{teo:forma-normale-2}. {Indeed, the main
  and first fact to notice is that if $q_0^*=0,\pi$ then $D_0=0$,
  since it depends only on sines; then by Lemma \ref{lem:M0.spec} we
  get $m_a(0,M_0)\geq 2$. Moreover, a direct computation shows that
  the algebraic multiplicity of the zero eigenvalue of $M_0$ is
  exactly two, so that we can apply Theorem~\ref{teo:forma-normale-2}. In order to verify the main condition
  \eqref{frm:stima-autovalori-bis}, since $D_0=0$}, we can restrict to
compute only $B_1$
$$ B_1 = \left( \begin{matrix} -2 & -1 & -1 \\ -1 & -2 & -1 \\ -1 & -1
  & -2 \\
\end{matrix}\right)T\ ,
$$ and we immediately obtain in all the four cases
\begin{displaymath}
  \gamma=\bigl\langle \langle B_1,\,{a}_1\rangle,{a}_1
  \bigr\rangle = -4T\not=0\ ,\qquad {a}_1=\partial_{\vartheta}Q_{j=1,3}\ .
\end{displaymath}

\begin{remark}
We stress that, for the in/out-of-phase configurations,
the true and approximate angles coincide, namely
$q_{\rm p.o.}=q^*$. This is due to the parity of the Hamiltonian in the
angles and to the Gauge symmetry; the first implies that the
remainder, whatever is its order in $\eps$, only depends on the
cosines, hence its ${p}$-field vanishes at any combination of
$0$ and $\pi$. The second implies that it does never depend on $q_1$,
being $p_1$ an exact constant of motion; in other words, the field
depends only on slow angles $q$. In this case,
Theorem~\ref{teo:forma-normale-2} could be simplified.
\end{remark}

It remains to investigate the second family $Q_2$, which satisfies the
necessary condition \eqref{frm:kernel-hp} simply because it represents a
solution for \eqref{frm:F(q,epsilon)}, namely $F(Q_2(\theta))\equiv 0\,$.

We explicitly construct the normal form up to order three by using
Mathematica\texttrademark and check that this family still persist.
This led us to conjecture that it represents a true solution of the
problem. Indeed, using the complex coordinates as
in~\eqref{frm:H-dNLS}, we can reformulate the continuation of
periodic orbits on the completely resonant torus $I=(I^*,I^*,I^*,I^*)$
by using the usual ansatz
\begin{displaymath}
\psi_j = e^{-\im\omega}\phi_j\ ,
\end{displaymath}
which provides the stationary equation for the amplitudes $\phi_j$
\begin{displaymath}
\lambda\phi_j = 2\phi_j|\phi_j|^2 + \eps
(L\phi)_j\ ,\qquad\lambda=\omega-1\ ,\qquad
(L\phi)_j=\phi_{j+1}+\phi_{j-1}\ .
\end{displaymath}
If we further assume that the continued solutions have the same
amplitude at all the sites, $|\phi_j|=a$, and the phase-shifts belong
to the second family $Q_2$
\begin{displaymath}
  \phi_j = a
  e^{\im\vphi_j}\ ,\qquad\vphi=(\vphi_1,\vphi_1+\theta,\vphi_1+\pi,\vphi_1+\theta+\pi)\ ,
\end{displaymath}
then we realize that for any $\theta\in S^1$ one has
\begin{displaymath}
  L e^{\im\vphi(\theta)}=0\ .
\end{displaymath}
Hence the stationary equation becomes
\begin{displaymath}
  \lambda = 2a^2 = 2I^*\ ,
\end{displaymath}
{which implies that a two-dimensional resonant torus,
  embedded in the original unperturbed four dimensional torus,
  survives for any given $\eps$.}

\begin{remark}
The above formulation would provide a much simpler proof for the
existence of the in/out-of-phase periodic orbits for $\eps\neq 0$, by
simply restricting to study the real $\phi$ configurations solving the
stationary equations \cite{HenT99, Kev_book09}. However, the role of
this example in the present paper is to show how the formal algorithm
works and what kind of insights can lead to in the investigation of
the breaking of completely resonant tori.
\end{remark}


\section{Conclusions}
\label{sec:concl}
Motivated by the aim of investigating the continuation of periodic
orbits on a completely resonant torus with respect to a small
parameter, we have built up an original normal form algorithm for a
classical Hamiltonian model of the form \eqref{frm:H-modello}. This
method naturally extends the averaging procedure of Poincar\'e, which
applies only to non-degenerate approximated solutions. {Hence, it
  allows to deal with all those cases when the extrema of the averaged
  Hamiltonian are not isolated}, like the one-parameter families
explored in Section~\ref{sec:applic}. {The present formulation of the
  result deals with the case of a maximal torus, hence it is more
  suitable for applications for few-bodies problems, e.g., in
  Celestial Mechanics.  In this field, the normal form construction
  here proposed, which provides an highly accurate approximate
  dynamics, could be effectively implemented with the aid of an
  algebraic manipulator (see, e.g., \cite{GioSan-2012}). Besides, the
  use of numerical tools could also include the analysis of the
  spectrum of $M(\eps)$}, which can be approximated at leading order
by the Floquet exponents of the approximate periodic orbits. Hence,
hypothesis \eqref{frm:stima-autovalori} can be verified numerically,
by tracking the dependence of the approximate Floquet spectrum on
$\eps$ in a neighbourhood of the origin.

The normal form algorithm here developed, if suitably extended to
completely resonant low-dimensional tori, could also allow to deal
with degenerate scenarios which emerge studying discrete solitons in
1D non-local discrete nonlinear Schroedinger lattices (like Zig-Zag
dNLS, see \cite{PenKSKP17}): in these models, one parameter families
of solutions of the averaged Hamiltonian appear when in the model long
range interactions (like next-to-nearest neighbourhood) are
added. More naturally, one parameter families of approximate
solutions, like the ones observed in the application developed in
Section~\ref{sec:applic}, appear in the investigation of vortexes in
2D square lattices \cite{PelKF05b}.  In these problems, the only
approach which has been till now explored and applied is based on
bifurcation methods \cite{PelKF05b,PenSPKK16} suitably combined with a
perturbation scheme. Hence, a different and completely constructive
approach would be desirable, especially in terms of possible
applications to the above mentioned lattice models with the help of a
manipulator. This further and not trivial extension will be worked out
in a future publication.


\appendix

\section{Technicalities: normal form construction}
\label{app:est}

The appendix is devoted to technical details and proofs related to the
normal form construction which have been moved here in order to avoid
the overloading of the text.

\subsection{Estimates for the $\nu_{r,s}$ sequence}

\begin{lemma}
The sequence $\{\nu_{r,s}\}_{r\ge 0\,,\,s\ge 0}$ defined in
\eqref{frm:nu-sequence} is bounded by the exponential growth
$$ \nu_{r,s}\leq\nu_{s,s}\leq \frac{100^{s}}{20} \qquad \hbox{for } \quad r\geq
0\,,\ s\geq 0\ .
$$
\end{lemma}
 
\begin{proof}
We start with the elimination of $\nu^{(\text{I})}_{r,s}$ in the
definition of $\nu_{r,s}$
$$
\begin{aligned}
\nu_{r,s}&= \sum_{j=0}^{\lfloor s/r
  \rfloor} (3\nu_{r-1,r})^{j}
\sum_{i=0}^{\lfloor s/r \rfloor-j}
(\nu_{r-1,r} )^{i}\nu_{r-1,s-(i+j)r} 
\\
&=\sum_{j=0}^{\lfloor s/r \rfloor} (3\nu_{r-1,r} )^{j}
\sum_{i=j}^{\lfloor s/r \rfloor}
(\nu_{r-1,r} )^{i-j}\nu_{r-1,s-ir}  \cr
&=\sum_{i=0}^{\lfloor s/r \rfloor}
(\nu_{r-1,r} )^{i}\nu_{r-1,s-ir}  \sum_{j=0}^{i}
3^{j} =\sum_{i=0}^{\lfloor s/r \rfloor} \frac{3^{i+1} -1}{2}
(\nu_{r-1,r} )^{i}\nu_{r-1,s-ir} \ , \cr
\end{aligned}
$$ where in the second equality we have exploited
$\nu_{r,0}^{(\text{I})}=1$. Thus we can rewrite the sequence as
\begin{equation*}
\nu_{r,s} = \sum_{j=0}^{\lfloor s/r
  \rfloor} \theta_j\nu_{r-1,r}^{j}\nu_{r-1,s-jr}\ ,
\qquad
\theta_{j} = \frac{3^{j+1} -1}{2}\ .
\end{equation*}
It is immediate to notice that $\nu_{r,s} \leq \nu_{s,s}$ for $s\geq
r$, hence
\begin{equation*}
\nu_{0,s}\le\nu_{1,s}\le\ldots\le\nu_{s,s}=\nu_{s+1,s}=\ldots\ .
\end{equation*}
Moreover
\begin{equation}
  \label{frm:stima-theta-j}
\theta_{0} = 1\ , \qquad \theta_{1} = 4\ , \qquad \theta_{j+1} \leq 5
\theta_{j}\qquad\hbox{for } j\geq0\ .
\end{equation}
and observing that
$\nu_{r,r}=\theta_{0}\nu_{r-1,r}+\theta_{1}\nu_{r-1,r}\,$,
 we get
\begin{equation} \label{frm:stima-nu_(r,r)}
\nu_{r,r}=5\nu_{r-1,r}\qquad\hbox{for }r\ge 1\ .
\end{equation}
From the definition of $\{\nu_{r,s}\}$, we can derive the following: for $r\geq
2$ and $s>2r$ we have
\begin{equation*}
\begin{aligned}
\nu_{r,s}&=\nu_{r-1,s}+ \nu_{r-1,r}\sum_{j=0}^{\lfloor s/r
  \rfloor-1}\theta_{j+1}\nu_{r-1,r}^{j}\nu_{r-1,s-r-jr} \cr
&\leq\nu_{r-1,s}+5\nu_{r-1,r} \sum_{j=0}^{\lfloor s/r
  \rfloor-1}\theta_{j}\nu_{r-1,r}^{j}\nu_{r-1,s-r-jr} \cr
&\leq\nu_{r-1,s}+5\nu_{r-1,r}\nu_{r,s-r}
\leq\nu_{r-1,s}+\nu_{r,r}\nu_{s-r,s-r}\ ,
\end{aligned}
\end{equation*}
where \eqref{frm:stima-theta-j} and \eqref{frm:stima-nu_(r,r)} have been
used; for $r=1$  we have 
\begin{equation*}
\begin{aligned}
\nu_{1,s}&=\nu_{0,s}+\nu_{0,1}\sum_{j=0}^{s-1}\theta_{j+1}\nu_{0,1}^{j}\nu_{0,s-1-j}
\cr &\leq(1+\theta_{1})\nu_{0,s-1}+
5\sum_{j=1}^{s-1}\theta_{j}\nu_{0,1}^{j}\nu_{0,s-1-j} \\
&\leq5\nu_{1,s-1}\leq 5\nu_{s-2,s-1} =\nu_{s-1,s-1}\ ,
\end{aligned}
\end{equation*}
where \eqref{frm:stima-theta-j} has been used, together with
$\nu_{0,s}=1$, for $s\ge 0$.  Due to the above
properties,  we can estimate
$\{\nu_{r,s}\}_{r\ge 0\,,\,s\ge 0}$ by means of its diagonal terms
$\nu_{r,r}$.  Indeed, $\nu_{1,1}=5$ and for $s>2$
$$
\begin{aligned}
\nu_{r,r}&=5\nu_{r-1,r}\leq5\nu_{r-2,r}+5\nu_{r-1,r-1}\nu_{1,1}
\leq\ldots \cr &\leq5\nu_{1,r}+5\left(
\nu_{2,2}\nu_{r-2,r-2}+\ldots+\nu_{r-1,r-1}\nu_{1,1}\right)
\leq5\sum_{j=1}^{r-1}\nu_{j,j}\nu_{r-j,r-j}\ .
\end{aligned}
$$
From this last upper bound, it is possible to verify
$$ \nu_{r,r}\leq 5^{2r-1}\lambda_{r} \qquad \hbox{for } r\geq 1\ ,
$$ with $\{\lambda_r\}_{r\ge 1}$ being the Catalan sequence,
which satisfies $\lambda_{r}\leq4^{r-1}$, thus
$$ \nu_{r,s}\leq\nu_{s,s}\leq \frac{100^{s}}{20} \qquad \hbox{for }\quad r\geq
0\,,\ s\geq 0\ .
$$
\end{proof}

\subsection{Estimates for multiple Poisson brackets}\label{sbs:stima-serie-di-Lie}
Some Cauchy estimates on the derivatives in the restricted domains
will be useful.

\begin{lemma}\label{lem:stima-derivata-Lie}
Let $d\in\reali$ such that $0<d<1$ and $g\in\Pset_{2l}$ be an analytic function with bounded norm
$\|g\|_{1}$.  Then one has
\begin{equation*}
\left\|\parder{g}{\hat{p}_j}\right\|_{1-d}\leq
\frac{\left\|g\right\|_{1}}{d\rho}\ ,
\qquad
\left\|\parder{g}{\hat{q}_j}\right\|_{1-d}\leq
\frac{\left\|g\right\|_{1}}{e d \sigma}\ ,
\end{equation*}
\end{lemma}

\begin{proof}
  Given $g$ as in~\frmref{frm:funz}, one has
  $$
  \begin{aligned}
    \left\|\parder{g}{\hat{p}_j}\right\|_{1-d} &\leq
    \sum_{{i\in\naturali^{n} \atop |i| = l}}
    \sum_{{\scriptstyle{k\in\ZZ^{n}}}}
    \frac{i_j}{\rho} |g_{i,k}| (1-d)^{l-1} \rho^l e^{|k|(1-d)\sigma}\cr
    &\leq
    \frac{1}{d\rho}
    \sum_{{i\in\naturali^{n} \atop |i| = l}}
    \sum_{{\scriptstyle{k\in\ZZ^{n}}}}
    |g_{i,k}| \rho^l e^{|k|\sigma} = \frac{\left\|g\right\|_{1}}{d\rho}\ ,
  \end{aligned}
  $$
  where we use the elementary inequality $m(\lambda-x)^{m-1}\leq
  {\lambda^m}/{x}$, for $0<x<\lambda$ and $m\geq1$.

  Similarly,
  $$
  \begin{aligned}
    \left\|\parder{g}{\hat{q}_j}\right\|_{1-d} &\leq
    \sum_{{i\in\naturali^{n} \atop |i| = l}}
    \sum_{{\scriptstyle{k\in\ZZ^{n}}}}
     |k_j|\,|g_{i,k}| (1-d)^l \rho^l e^{|k|(1-d)\sigma}\cr
    &\leq
    \frac{1}{e d \sigma}
    \sum_{{i\in\naturali^{n} \atop |i| = l}}
    \sum_{{\scriptstyle{k\in\ZZ^{n}}}}
    |g_{i,k}| \rho^l e^{|k|\sigma} = \frac{\left\|g\right\|_{1}}{ed\sigma}\ ,
  \end{aligned}
  $$

  where we use the elementary inequality $x^\alpha e^{-\delta x}\leq
  \left({\alpha}/({e \delta})\right)^\alpha$, for positive $\alpha$, $x$ and $\delta$.
\end{proof}

\begin{lemma}\label{lem:stima-derivata-gen}
Let $d\in\reali$ such that $0<d<1$.  Let the generating functions
$\chi_0^{(r)}$ and $\chi_2^{(r)}$ be as in~\frmref{frm:chi0-r}.  Then
one has
\begin{align}
\Biggl\|\parder{\chi_0^{(r)}}{\hat{q}_j}\Biggr\|_{1-d}&\leq
\frac{\bigl\|X_0^{(r)}\bigr\|_{1}}{ed\sigma} + |\zeta^{(r)}|\ ,
\label{frm:stima-dchi0dq}
\\
\left\|\parder{\chi_2^{(r)}}{\hat{q}_j}\right\|_{1-d}&\leq
\frac{\bigl\|\chi_2^{(r)}\bigr\|_{1}}{ed\sigma}\ ,
\label{frm:stima-dchi2dq}
\\
\left\|\parder{\chi_2^{(r)}}{\hat{p}_j}\right\|_{1-d}&\leq
\frac{\bigl\|\chi_2^{(r)}\bigr\|_{1}}{\rho}\ ;
\label{frm:stima-dchi2dp}
\end{align}
moreover, for $j\geq1$, 
\begin{align}
\left\|\lie{\chi_0^{(r)}}^{j}f\right\|_{1-d-d^{\prime}}
&\leq\frac{j!}{e}
\left(
\frac{\|X_0^{(r)}\|_{1-d'}}{{d^2\rho\sigma}} + \frac{e |\zeta^{(r)}|}{{d\rho}}
  \right)^{j}
  \|f\|_{1-d^{\prime}}\ ,
  \label{frm:stima-liechi0}
\\
\left\|\lie{\chi_2^{(r)}}^{j}f\right\|_{1-d-d^{\prime}}
&\leq\frac{j!}{e}
\left(
\frac{\|\chi_2^{(r)}\|_{1-d'}}{{d^2\rho\sigma}}
  \right)^{j}
  \|f\|_{1-d^{\prime}}\ ,
    \label{frm:stima-liechi2}
\end{align}
\end{lemma}

\begin{proof}
The proofs of~\eqref{frm:stima-dchi0dq}--\eqref{frm:stima-dchi2dp} are
just minor modifications of Lemma~\lemref{lem:stima-derivata-Lie},
thus are left to the reader.

Coming to~\eqref{frm:stima-liechi0}, let $\delta=d/j$ with $j\geq1$. Proceeding iteratively we get
\begin{align*}
\left\|\lie{\chi_0^{(r)}}^{j}f\right\|_{1-d-d^{\prime}}
&\leq
\left(
\frac{\|X_0^{(r)}\|_{1-d'}}{{j\delta^2e\rho\sigma}} + \frac{|\zeta^{(r)}|}{{\delta\rho}}
  \right)
\left\|\lie{\chi_0^{(r)}}^{j-1}f\right\|_{1-d^{\prime}-(j-1)\delta}
\cr
&\leq\ldots
\cr
&\leq\frac{j!}{e}
\left(
\frac{\|X_0^{(r)}\|_{1-d'}}{{d^2\rho\sigma}} + \frac{e |\zeta^{(r)}|}{{d\rho}}
  \right)^{j}\|f\|_{1-d^{\prime}}\ ,
\end{align*}
where we used the trivial inequality $j^{j}\leq
j!\,e^{j-1}$, holding true for $j\ge 1\,$.
Finally, the proof of~\eqref{frm:stima-liechi2} is the same, mutatis mutandis.
\end{proof}

\subsection{Estimates for the generating functions}\label{sbs:stima-generatrici}
\begin{lemma}
\label{lem:generatrici}
Let $d\in\reali$ such that $0<d<1$. The generating function
$X_0^{(r)}$ and the vector $\zeta^{(r)}$ are bounded by
\begin{equation}
\Vert X_0^{(r)}\Vert_{1-d} \leq \frac{\Vert
  f_0^{(r-1,r)}\Vert_{1-d}}{\omega}\ , \qquad
\vert\zeta^{(r)}\vert \leq \frac{\Vert f_2^{(r-1,r)}\Vert_{1-d}}{m\rho}\ .
\label{frm:stima-chi0}
\end{equation}
The generating function $\chi_2^{(r)}$ is instead bounded by
\begin{equation}
  \Vert\chi_2^{(r)}\Vert_{1-d} \leq
  \frac{1}{\omega}\left(
  2\Vert f_2^{(r-1,r)}\Vert_{1-d} + \frac{2}{e\delta_r\rho\sigma}
  \frac{\Vert f_0^{(r-1,r)}\Vert_{1-d}}{\omega} \|f_4^{(0,0)}\|_{1}
  \right)\ .
\label{frm:stima-chi2}
\end{equation}
\end{lemma}

\begin{proof}
The estimate for $X_0^{(r)}$ is trivial.  The estimate for
$\chi_2^{(r)}$, that is controlled by $f_2^{(\rmI;r-1,r)}$, is a
little bit tricky.  Indeed, one has to explicitly exploit the fact that
$$
f^{(\rmI;r-1,r)}_{2} = f^{(r-1,r)}_{2} - \inter{f^{(r-1,r)}_{2}(q^*)}_{q_1}
+ L_{X_0^{(r)}} f_4^{(0,0)}\ ,
$$
together with the trivial estimate 
$$
\Vert f-\langle f(q^*)\rangle_{q_1} \Vert_{1-d} \leq 2\Vert f\Vert_{1-d}\ .
$$
As $C$ satisfy~\eqref{frm:twist}, there exists a solution $\zeta^{(r)}$
of~\eqref{frm:traslazione} which satisfies
$$
\left\|\nabla_{\hat{p}}\bigl\langle f_{2}^{(r-1,r)}\bigr|_{{q}={q}^*}\bigr\rangle_{q_1}\right\|_{1-d_{r-1}}
=
\Bigl|\sum_{j} C_{ij} \zeta_j^{(r)}\Bigr| \geq m |\zeta^{(r)}|
\ .
$$
Moreover, by the definition of the norm one has
$$
\left\|\nabla_{\hat{p}}\bigl\langle f_{2}^{(r-1,r)}\bigr|_{{q}={q}^*}\bigr\rangle_{q_1}\right\|_{1-d_{r-1}}
=
\frac{\left\|\bigl\langle f_{2}^{(r-1,r)}\bigr|_{{q}={q}^*}\bigr\rangle_{q_1}\right\|_{1-d_{r-1}}}{\rho}
\leq
\frac{\left\| f_{2}^{(r-1,r)}\right\|_{1-d_{r-1}}}{\rho}
\ .
$$
Combining the latter inequalities one gets~\eqref{frm:stima-chi0}.
\end{proof}

\subsection{Estimates for the first and second normalization step}\label{sbs:stima-primosecondo-passo}
The following two Lemmas collect the estimates concerning the first
two steps of the normal form algorithm previously described.  We
decide to explicitly report the results concerning the normal form at
order one and two with the purpose of making transparent the structure
of the estimates of the different terms appearing in the normalized
Hamiltonian.  Furthermore, the first two steps are needed so as to
verify the inductive proof for the forthcoming
Lemma~\lemref{lem:lemmone}.

\begin{lemma} \label{lem:lemmone.1}
  Consider a Hamiltonian $H^{(0)}$ expanded as in~\frmref{frm:H(0)-f}.
  Let $\chi_0^{(1)}$ and $\chi_2^{(1)}$ be the generating functions
  used to put the Hamiltonian in normal form at order one, then one has
  \begin{equation*}
  \begin{aligned}
    \|X_0^{(1)}\|_{1} &\leq \frac{1}{\omega} \nu_{0,1} E\epsilon\ ,\cr
    |\zeta^{(1)}| &\leq  \frac{1}{4m\rho} \nu_{0,1} E\epsilon\ ,\cr
    \|\chi_2^{(1)}\|_{1-\delta_1} &\leq \frac{1}{\omega}
    3\nu_{0,1} \Xi_1  \frac{E}{4}\epsilon\ .
  \end{aligned}
  \end{equation*}
  The terms appearing in the expansion of $H^{(\rmI;0)}$, i.e.
  in~\frmref{frm:H(I;r-1)-f} with $r=1$, are bounded as
  $$
  \begin{aligned}
    \|f_{0}^{(\rm I;0,1)}\|_{1-\delta_1} &\leq {E}\epsilon\ ,\cr
    \|f_{l}^{(\rm I;0,s)}\|_{1-\delta_1} &\leq \nu^{(\rm I)}_{1,s} \Xi_1^s 
    \frac{E}{2^l}\epsilon^s\ .
  \end{aligned}
  $$
  The terms appearing in the expansion of $H^{(1)}$, i.e.
  in~\frmref{frm:f^(r,s)} with $r=1$, are bounded as
  \begin{equation}
  \begin{aligned}
    \|f_{0}^{(1,s)}\|_{1-d_1} &\leq \nu_{1,s} \Xi_1^{2s-2} {E}\epsilon^s\ ,\cr
    \|f_{2}^{(1,s)}\|_{1-d_1} &\leq \nu_{1,s} \Xi_1^{2s-1} \frac{E}{2^2}\epsilon^s\ ,\cr
    \|f_{l}^{(1,s)}\|_{1-d_1} &\leq \nu_{1,s} \Xi_1^{2s}  \frac{E}{2^l}\epsilon^s\ .
  \end{aligned}
  \end{equation}
\end{lemma}

\begin{proof}
Using Lemma~\ref{lem:generatrici}, we immediately get the bounds
$$
\norm{X_0^{(1)}}_1\leq \frac1\omega \norm{f_0^{(0,1)}}_1 \leq
\frac1\omega E\eps\ ,\qquad |\zeta^{(1)}|\leq
\frac1{m\rho}\norm{f_2^{(0,1)}}_1 \leq \frac{E\eps}{4m\rho}\ ,
$$
thus, from~\eqref{frm:stima-dchi0dq} with $r=1$ we get
$$
\Biggl\|\parder{\chi_0^{(1)}}{\hat{q}_j}\Biggr\|_{1-\delta_1}
\leq
\frac{E \epsilon}{\omega e \delta_1 \sigma}
+\frac{E\eps}{4m\rho}
\leq
\left(\frac{1}{\omega e \delta_1 \sigma}
+\frac{1}{4m\rho}\right)
E \epsilon
\ .
$$

The terms $f_{l}^{(\text{I};0,s)}$ appearing in the expansion of the
Hamiltonian $H^{(\text{I};0)}$ are bounded as follows.  For $l=0$ and
$s=1$ one has
\begin{equation}
\|f_{0}^{(\text{I};0,1)}\|_{1-\delta_1} \leq
\|f_{0}^{(0,1)}\|_{1-\delta_1} \leq E\epsilon\ ,
\label{frm:f^{(I;0,1)}_0}
\end{equation}
while for the remaining terms one has
$$
\begin{aligned}
\|f_{l}^{(\text{I};0,s)}\|_{1-\delta_1} &\leq \sum_{j=0}^{s}
\frac{1}{j!} \|L_{\chi^{(1)}_{0}}^{j}
f^{(0,s-j)}_{l+2j}\|_{1-\delta_1}\cr &\leq \sum_{j=0}^{s} \frac{1}{e}
\left(
\frac{\|X_0^{(1)}\|_{1-d}}{{\delta_1^2\rho\sigma}} + \frac{e |\zeta^{(1)}|}{{\delta_1\rho}}
  \right)^{j}
  \|f^{(0,s-j)}_{l+2j}\|_{1}\cr &
  \leq \sum_{j=0}^{s} \frac{1}{e}
\left(\frac{1}{\omega \delta_1^2\rho\sigma }
+\frac{e}{4m\delta_1\rho^2}\right)^j
E^j \epsilon^j \frac{E}{2^{l+2j}}\epsilon^{s-j}\cr&\leq
\frac{E\epsilon^{s}}{2^{l}}\sum_{j=0}^{s} \frac{1}{e}
\left(\frac{E}{\omega \delta_1^2\rho\sigma }
+\frac{e E}{4m\delta_1\rho^2}\right)^j
\cr
&< (s+1) \Xi_1^s \frac{E}{2^l}\epsilon^s = \nu^{(\text{I})}_{1,s} \Xi_1^s
\frac{E}{2^l}\epsilon^s,
\end{aligned}
$$
where we used the definition of the constant $\Xi_1$ and 
Lemma~\ref{lem:stima-derivata-gen}.

Coming to the second stage of the normalization step, the generating
function $\chi_2^{(1)}$ is bounded, as in \eqref{frm:stima-chi2}, by
$$
\begin{aligned}
\|\chi_2^{(1)}\|_{1-\delta_1} &\leq \frac{1}{\omega}
\left(
  2 \Vert f_2^{(0,1)}\Vert_{1} + \frac{2}{e\delta_1\rho\sigma}
  \frac{\Vert f_0^{(0,1)}\Vert_{1-\delta_1}}{\omega} \|f_4^{(0,0)}\|_{1}
  \right)\\
& \leq \frac{1}{\omega}
\left(
  2 \frac{E}{4}\epsilon + \frac{2}{e\delta_1\rho\sigma}
  \frac{E \epsilon}{\omega} \frac{E}{2^4}
  \right)
  \\
  &\leq
   \frac{1}{\omega}
\left(
  2 + \frac{E}{2\omega e\delta_1\rho\sigma}
  \right) \frac{E}{4}\epsilon\\
  &<\frac{1}{\omega}3\nu_{0,1} \Xi_1 \frac{E}{4}\epsilon\ .
\end{aligned}
$$

The terms $f_{l}^{(1,s)}$ appearing in the expansion of the
Hamiltonian $H^{(1)}$ are bounded as follows.  The term $f_0^{(1,1)}$
is unchanged, while for $l=0$ and $s=2$ one has
$$
\begin{aligned}
  \|f_{0}^{(1,2)}\|_{1-d_1}
  &\leq \|f^{(\text{I};0,2)}_{0}\|_{1-\delta_1} +
  \frac{1}{e}\frac{1}{\delta_1^2\rho\sigma}
  \|{\chi^{(1)}_{2}}\|_{1-\delta_1}
  \|f^{(\text{I};0,1)}_{0}\|_{1-\delta_1}
  \cr
  &\leq
  \nu^{(\text{I})}_{1,2} \Xi_1^2  E \epsilon^2 +
  \frac{1}{e}\frac{1}{\delta_1^2\rho\sigma}
  \frac{1}{\omega}3\nu_{0,1} \Xi_1 \frac{E}{4}\epsilon
  E\epsilon
  \cr
  &\leq \nu_{1,2} \Xi_1^2 {E} \epsilon^2\ .
\end{aligned}
$$

For $l=0$ and $s\geq 3$, using~\frmref{frm:f^{(I;0,1)}_0} for the
estimate of the last term in the sum, one has
$$
\begin{aligned}
\|f_{0}^{(1,s)}\|_{1-d_1} &\leq \sum_{j=0}^{s-2} \frac{1}{e}
\left(\frac{1}{\delta_1^2\rho\sigma}\right)^j
\|{\chi^{(1)}_{2}}\|_{1-\delta_1}^{j}
\|f^{(\text{I};0,s-j)}_{0}\|_{1-\delta_1} \cr
&\qquad+\frac{1}{e}
\left(
\frac{1}{\delta_1^2\rho\sigma}
\right)^{s-1}
\|{\chi^{(1)}_{2}}\|_{1-\delta_1}^{s-1}
\|f^{(\text{I};0,1)}_{0}\|_{1-\delta_1}\cr
&\leq
\sum_{j=0}^{s-2} \frac{1}{e}
\left(\frac{1}{\delta_1^2\rho\sigma}\right)^j
\frac{1}{\omega^j}(3\nu_{0,1})^j 
\Xi_1^{j} \frac{E^j}{4^j} \epsilon^j
\nu^{(\text{I})}_{1,s-j} \Xi_1^{s-j} E \epsilon^{s-j}
\cr
&\qquad+
\frac{1}{e} \left(
\frac{1}{\delta_1^2\rho\sigma}
\right)^{s-1}
\frac{1}{\omega^{s-1}}(3\nu_{0,1})^{s-1} 
\Xi_1^{s-1} \frac{E^{s-1}}{4^{s-1}} \epsilon^{s-1}
E \epsilon\cr
 &\leq\nu_{1,s} \Xi_1^{2s-2} {E}
    \epsilon^s\ .
\end{aligned}
$$

The term $f_2^{(1,1)}$ is unchanged, while for $l=2$ and $s\geq 2$ one
has
$$
\begin{aligned}
  \|f_{2}^{(1,s)}\|_{1-d_1} &\leq
\sum_{j=0}^{s-2}
\frac{1}{e}
\left(\frac{1}{\delta_1^2\rho\sigma}\right)^j
\|{\chi^{(1)}_{2}}\|_{1-\delta_1}^{j}
\|f^{(\rmI;0,s-j)}_{2}\|_{1-\delta_1}\cr
& \quad+\frac{1}{e}
\left(\frac{1}{\delta_1^2\rho\sigma}\right)^{s-1}
\|{\chi^{(1)}_{2}}\|_{1-\delta_1}^{s-1}
\|f^{(\text{I};0,1)}_{2}\|_{1-\delta_1} +\cr
&\leq
\sum_{j=0}^{s-2}
  \frac{1}{e} \left(\frac{1}{\delta_1^2\rho\sigma}\right)^j
  \frac{1}{\omega^j}(3\nu_{0,1})^j \Xi_1^{j}
  \frac{E^j}{4^j}\epsilon^{j} \nu^{(\text{I})}_{1,s-j}
  \Xi_1^{s-j} \frac{E}{4}\epsilon^{s-j}\cr
&\quad+  
  \frac{1}{e}\left(\frac{1}{\delta_1^2\rho\sigma}\right)^{s-1}
  \frac{1}{\omega^{s-1}}(3\nu_{0,1})^{s-1} \Xi_1^{s-1}
  \frac{E^{s-1}}{4^{s-1}}\epsilon^{s-1} \nu^{(\text{I})}_{1,1}
  \Xi_1 \frac{E}{4}\epsilon \cr
  &\leq
\nu_{1,s} \Xi_1^{2s-1}
    \frac{E}{2^2} \epsilon^s\ .
\end{aligned}
$$

Finally, for $l>2$ and $s\geq 1$ one has
$$
\begin{aligned}
\|f_{l}^{(1,s)}\|_{1-d_1} &\leq  \sum_{j=0}^{s} \frac{1}{e}
\left(\frac{1}{\delta_1^2\rho\sigma}\right)^j
\|{\chi^{(1)}_{2}}\|_{1-\delta_1}^{j}
\|f^{(\rmI;0,s-j)}_{l}\|_{1-\delta_1}\cr &\leq \sum_{j=0}^{s}
\frac{1}{e}
\left(\frac{1}{\delta_1^2\rho\sigma}\right)^j
\frac{1}{\omega^j}
(3\nu_{0,1})^j \Xi_1^{j} \frac{E^j}{4^j}
\epsilon^{j} \nu^{(\text{I})}_{1,s-j} \Xi_1^{(s-j)}
\frac{E}{2^{l}}\epsilon^{s-j} \cr
&\leq
\nu_{1,s} \Xi_1^{2s} \frac{E}{2^{l}}
  \epsilon^s\ .
\end{aligned}
$$
This concludes the proof of the Lemma.
\end{proof}

\begin{lemma}\label{lem:lemmone.2}
  Consider a Hamiltonian $H^{(1)}$ expanded as in~\frmref{frm:H(r-1)-f}.
  Let $\chi_0^{(2)}$ and $\chi_2^{(2)}$ be the generating functions
  used to put the Hamiltonian in normal form at order two, then one has
  \begin{equation*}
  \begin{aligned}
    \|X_0^{(2)}\|_{1-d_1} &\leq \frac{1}{\omega} \nu_{1,2} \Xi_2^2 E\epsilon^2\ ,\cr
    |\zeta^{(2)}| &\leq  \frac{1}{4m\rho} \nu_{1,2} \Xi_2^3 E\epsilon^2\ ,\cr
    \|\chi_2^{(2)}\|_{1-d_1-\delta_2} &\leq \frac{1}{\omega}
     3\nu_{1,2} \Xi_2^3 \frac{E}{4}\epsilon^2\ .
  \end{aligned}
  \end{equation*}
  The terms appearing in the expansion of $H^{(\rmI;1)}$, i.e.
  in~\frmref{frm:H(I;r-1)-f} with $r=2$, are bounded as
  \begin{equation*}
  \begin{aligned}
    \|f_{0}^{(\rmI;1,s)}\|_{1-d_1-\delta_2} &\leq  \nu^{(\rm I)}_{2,s}\Xi_2^{2s-2}  {E}\epsilon^s\ , &\qquad &\hbox{for }1\leq s\leq 2\ ,\cr
    \|f_{0}^{(\rmI;1,s)}\|_{1-d_1-\delta_2} &\leq  \nu^{(\rm I)}_{2,s}\Xi_2^{2s-1}  {E}\epsilon^s\ , &\qquad &\hbox{for }2< s\leq 4\ ,\cr
    \|f_{2}^{(\rmI;1,s)}\|_{1-d_1-\delta_2} &\leq  \nu^{(\rm I)}_{2,s}\Xi_2^{2s-1}  \frac{E}{2}\epsilon^s\ , &\qquad &\hbox{for }1\leq s\leq 2\ ,\cr
    \|f_{l}^{(\rmI;1,s)}\|_{1-d_1-\delta_2} &\leq  \nu^{(\rm I)}_{2,s}\Xi_2^{2s}  \frac{E}{2^l}\epsilon^s\ ,&\qquad &\hbox{for the remaining cases}\ .
  \end{aligned}
  \end{equation*}
  The terms appearing in the expansion of $H^{(2)}$, i.e.
  in~\frmref{frm:f^(r,s)} with $r=2$, are bounded as
  \begin{equation*}
  \begin{aligned}
    \|f_{0}^{(2,s)}\|_{1-d_2} &\leq  \nu_{2,s}\Xi_2^{2s-2}  {E}\epsilon^s\ , &\qquad &\hbox{for }1\leq s\leq 2\ ,\cr
    \|f_{0}^{(2,s)}\|_{1-d_2} &\leq  \nu_{2,s}\Xi_2^{2s-1}  {E}\epsilon^s\ , &\qquad &\hbox{for }2< s\leq 4\ ,\cr
    \|f_{2}^{(2,s)}\|_{1-d_2} &\leq  \nu_{2,s}\Xi_2^{2s-1} \frac{E}{2}\epsilon^s\ , &\qquad &\hbox{for }1\leq s\leq 2\ ,\cr
    \|f_{l}^{(2,s)}\|_{1-d_2} &\leq  \nu_{2,s}\Xi_2^{2s}   \frac{E}{2^l}\epsilon^s\ ,&\qquad &\hbox{for the remaining cases}\ .
  \end{aligned}
  \end{equation*}
\end{lemma}

\begin{proof}
Using Lemma~\ref{lem:generatrici} and the estimates in
Lemma~\ref{lem:lemmone.1}, we immediately get
$$
\norm{X_0^{(2)}}_{1-d_1}\leq 
\frac{1}{\omega} \nu_{1,2} \Xi_2^2  E \epsilon^2\ ,\qquad
|\zeta^{(2)}|\leq
\frac{1}{m\rho} \nu_{1,2} \Xi_2^3 \frac{E}{4}\eps^2\ ,
$$
thus, from~\eqref{frm:stima-dchi0dq} we get
$$
\Biggl\|\parder{\chi_0^{(2)}}{\hat{q}_j}\Biggr\|_{1-d_1-\delta_2}
\leq
\frac{1}{\omega e \delta_2 \sigma} \nu_{1,2} \Xi_2^2 E \epsilon^2
+\frac{1}{4 m\rho} \nu_{1,2} \Xi_2^3 E\eps^2
\leq
\left(\frac{1}{\omega e \delta_2 \sigma}
+\frac{1}{4m\rho}\right)
\nu_{1,2} \Xi_2^3 E \epsilon^2
\ ,
$$

The terms $f_{l}^{(\rmI;1,s)}$ appearing in the expansion of the Hamiltonian
$H^{(\rmI;1)}$ are bounded as follows.  For $s=1$ all the terms are unchanged,
thus there is nothing to do.  Furthermore notice that $f_{0}^{(\rmI;1,2)}$ is
trivially bounded with the norm of $f_{0}^{(1,2)}$.
The term $f_{2}^{(\rmI;1,2)}$ requires more care, indeed
$$
f^{(\rmI;1,2)}_{2} = f^{(1,2)}_{2} - \inter{f^{(1,2)}_{2}(q^*)}_{q_1}
+ L_{X_0^{(2)}} f_4^{(0,0)}\ .
$$ Thus only the generating function $X_0^{(2)}$ plays a role and we
get the following estimate
$$ \|f_{2}^{(\text{I};1,2)}\|_{1-d_1-\delta_2} \leq 2 \nu_{1,2}
\Xi_2^3 \frac{E}{4} \epsilon^2 + \frac{1}{\omega e \delta_2 \rho
  \sigma} \nu_{1,2} \Xi_2^2 E \epsilon^2 \frac{E}{4} \leq 3\nu_{1,2}
\Xi_2^3 \frac{E}{4} \epsilon^2< \nu_{2,2}^{(\rm I)} \Xi_2^3 \frac{E}{2}
\epsilon^2\ .
$$
For $l=0$ and $s=3$ one has
$$
\begin{aligned}
\|f_{0}^{(\rmI;1,3)}\|_{1-d_1-\delta_2} &\leq
\|f_{0}^{(1,3)}\|_{1-d_1-\delta_2} + \|L_{\chi^{(2)}_{0}} f_{2}^{(1,1)}\|_{1-d_1-\delta_2}
\leq \nu_{2,3}^{(\text{I})} \Xi_2^4 E \epsilon^3
\end{aligned}
$$
Similarly, for $l=0$ and $s=4$ one has
$$
\begin{aligned}
\|f_{0}^{(\rmI;1,4)}\|_{1-d_1-\delta_2} &\leq
\|f_{0}^{(1,4)}\|_{1-d_1-\delta_2} + \|L_{\chi^{(2)}_{0}} f_{2}^{(1,2)}\|_{1-d_1-\delta_2}+ \|L_{\chi^{(2)}_{0}}^2 f_{4}^{(1,0)}\|_{1-d_1-\delta_2}
\leq \nu_{2,4}^{(\text{I})} \Xi_2^7 E \epsilon^4\ .
\end{aligned}
$$
For the remaining terms one has
$$
\begin{aligned}
\|f_{l}^{(\text{I};1,s)}\|_{1-d_1-\delta_2} &\leq \sum_{j=0}^{\lfloor
  s/2\rfloor} \frac{1}{j!} \|L_{\chi^{(2)}_{0}}^{j}
f^{(1,s-2j)}_{l+2j}\|_{1-d_1-\delta_2}\cr &\leq \sum_{j=0}^{\lfloor
  s/2\rfloor}
\frac{1}{e}
\left(
\frac{\|X_0^{(2)}\|_{1-d_1}}{{\delta_2^2\rho\sigma}} + \frac{e |\zeta^{(2)}|}{{\delta_2\rho}}
  \right)^{j} \|f^{(1,s-2j)}_{l+2j}\|_{1-d_1}\cr
  &
  \leq \sum_{j=0}^{\lfloor s/2\rfloor}
\frac{1}{e}
\left(\frac{1}{\omega \delta_2^2\rho\sigma }
+\frac{e}{4m\delta_2\rho^2}\right)^j
\frac{1}{\omega^j}\nu_{1,2}^j \Xi_2^{3j} E^j \epsilon^{2j}
\nu_{1,s-2j} \Xi_2^{2(s-2j)}\frac{E}{2^{l+2j}}\epsilon^{s-2j}\cr
&\leq
\nu^{(\text{I})}_{2,s} \Xi_2^{2s}  \frac{E}{2^l}\epsilon^s\ .
\end{aligned}
$$

Coming to the second stage of the normalization step, the generating
function $\chi_2^{(2)}$ is bounded by
$$
\|\chi_2^{(2)}\|_{1-d_1-\delta_2} \leq \frac{1}{\omega}
\|f_2^{(\text{I};1,2)}\|_{1-d_1-\delta_2} \leq
\frac{1}{\omega} 3\nu_{1,2} \Xi_2^{3} \frac{E}{4} \epsilon^2\ .
$$
The terms $f_{l}^{(2,s)}$ appearing in the expansion of the
Hamiltonian $H^{(2)}$ are bounded as follows.  For $s=1$ all the terms
are unchanged, thus there is nothing to do.  Furthermore both
$f_{0}^{(2,2)}$ and $f_{2}^{(2,2)}$ are trivially bounded with the
norm of $f_{0}^{(\rmI,1,2)}$ and $f_{2}^{(\rmI,1,2)}$, respectively.
Similarly to the first stage of the the normalization step, the terms
$f_{0}^{(2,3)}$ and $f_{0}^{(2,4)}$ are bounded as follows
$$
\begin{aligned}
\|f_{0}^{(2,3)}\|_{1-d_2} &\leq \nu_{2,3} \Xi_2^5 E \epsilon^3\ ,\qquad
\|f_{0}^{(2,4)}\|_{1-d_2} &\leq \nu_{2,4} \Xi_2^7 E \epsilon^4\ .
\end{aligned}
$$
For the remaining terms one has
$$
\begin{aligned}
\|f_{l}^{(2,s)}\|_{1-d_2} &\leq \sum_{j=0}^{[s/2]}
\frac{1}{e} \left(\frac{1}{\delta_2^2 \rho\sigma}\right)^j
\frac{1}{\omega^j} (3 \nu_{1,2})^j
\Xi_2^{3j} \frac{E^j}{2^{2j}} \epsilon^{2j}
\nu_{2,s-2j}^{(\rmI)} \Xi_2^{2s-4j}
\frac{E}{2^l}\epsilon^{s-2j}\leq \nu_{2,s} \Xi_2^{2s}   \frac{E}{2^l}\epsilon^{s}\ .
\end{aligned}
$$
This concludes the proof of the Lemma.
\end{proof}

\begin{lemma}\label{lem:parti-intere}
Let $s = \lfloor s/r\rfloor r + m$, then for $0\leq j \leq  \lfloor s/r\rfloor$ one has
$$
3rj-2j+b(r-1,s-jr,l+2j) \leq b(\rmI;r-1,s,l)\ .
$$
\end{lemma}

\begin{proof}
The proof just requires a trivial computation, i.e.,
$$
\begin{aligned}
3rj-2j&+b(r-1,s-jr,l+2j) =\cr &=3rj-2j+3(s-jr)-\biggl\lfloor
\frac{s-jr+r-2}{r-1}\biggr\rfloor-\biggl\lfloor
\frac{s-jr+r-3}{r-1}\biggr\rfloor\cr &=3s-\biggl\lfloor
\frac{s-j+r-2}{r-1}\biggr\rfloor-\biggl\lfloor
\frac{s-j+r-3}{r-1}\biggr\rfloor\cr &=3s-\biggl\lfloor \frac{\lfloor
  s/r\rfloor r + m -j+r-2}{r-1}\biggr\rfloor -\biggl\lfloor
\frac{\lfloor s/r\rfloor r + m-j+r-3}{r-1}\biggr\rfloor\cr
&=3s-\biggl\lfloor \frac{s}{r} \biggr\rfloor -\biggl\lfloor
\frac{\lfloor s/r\rfloor + m -j+r-2}{r-1}\biggr\rfloor -\biggl\lfloor
\frac{s}{r} \biggr\rfloor -\biggl\lfloor \frac{\lfloor s/r\rfloor + m
  -j+r-3}{r-1}\biggr\rfloor\cr
&\leq3s-\biggl\lfloor \frac{s+r-1}{r}\biggr\rfloor-\biggl\lfloor \frac{s+r-2}{r}\biggr\rfloor\cr
&\leq b(\text{I};r-1,s,l)
\end{aligned}
$$
\end{proof}

\subsubsection{Proof of lemma \ref{lem:lemmone}}\label{sbs:proof-lemmone}
We proceed by induction.  For $r=1,2$ just use Lemmas~\lemref{lem:lemmone.1} and
~\lemref{lem:lemmone.2}, respectively.

For $r>2$, the estimates~\eqref{frm:chi-estimates} for the generating functions
follow directly from Lemma~\ref{lem:generatrici}, remarking that
$$
b(r-1,r,2) =
b(\rmI;r-1,r,2) =
3r-\left\lfloor\frac{2r-2}{r-1}\right\rfloor-\left\lfloor\frac{2r-3}{r-1}\right\rfloor =
3r-3\ .
$$

The terms $f_{l}^{(\rmI;r-1,s)}$ appearing in the expansion of the
Hamiltonian $H^{(\rmI;r-1)}$ are bounded as follows.  For $l=0,1$ and
$s< r$ all the terms are unchanged, thus there is nothing to do.  The
term $f_{0}^{(\rmI;r-1,r)}$ is trivially bounded with the norm of
$f_{0}^{(r-1,r)}$.  The term $f_{2}^{(\rmI;r-1,r)}$ requires more
care\footnote{See the proofs of Lemma \ref{lem:lemmone.2} and Lemma
  \ref{lem:generatrici}.} since only the generating function
$X_0^{(r)}$ plays a role and we get the following estimate
$$
\|f_{2}^{(\text{I};r-1,r)}\|_{1-d_{r-1}-\delta_r} \leq
3\nu_{r-1,r} \Xi_r^{3r-3} \frac{E}{4} \epsilon^r\ .
$$

For $l=0$ and $r<s\leq 2r$, 

$$
\begin{aligned}
\|f_{0}^{(\text{I};r-1,s)}\|_{1-d_{r-1}-\delta_r} &\leq
\|f^{(r-1,s)}_{0}\|_{1-d_{r-1}-\delta_r}+\|\lie{\chi^{(r)}_{0}}f^{(r-1,s-r)}_{2}\|_{1-d_{r-1}-\delta_r}\cr
&\leq \nu_{r-1,s} \Xi_r^{b(r-1,s,0)}  E \epsilon^{s} \cr&\qquad+
\frac{1}{e}
\left(
\frac{\|X_0^{(r)}\|_{1-d_{r-1}}}{{\delta_r^2\rho\sigma}} + \frac{e |\zeta^{(r)}|}{{\delta_r\rho}}
\right)
\nu_{r-1,s-r}
\Xi_r^{b(r-1,s-r,2)} 
\frac{E}{2^{2}}\epsilon^{s-r}\cr
&\leq \nu_{r-1,s} \Xi_r^{b(r-1,s,0)}  E \epsilon^{s} \cr&\qquad+
\frac{1}{e}
\left(
\frac{E}{\omega\delta_r^2\rho\sigma} + \frac{e E}{{4m\delta_r\rho}}\right) \nu_{r-1,r} \Xi_r^{b(r-1,r,2)}
\epsilon^{r}
\nu_{r-1,s-r}
\Xi_r^{b(r-1,s-r,2)} 
\frac{E}{2^{2}}\epsilon^{s-r}\cr
&\leq 
\Xi_r^{b(\text{I};r-1,s,0)} \nu^{(\text{I})}_{r,s} E\epsilon^s\ ,
\end{aligned}
$$
where we use the trivial inequality
$$
3r-2+b(r-1,s-r,2) \leq b(\rmI;r-1,s,0)\ .
$$

For the remaining terms one has
$$
\begin{aligned}
\|f_{l}^{(\text{I};r-1,s)}\|_{1-d_{r-1}-\delta_r} &\leq
\sum_{j=0}^{\lfloor s/r\rfloor} \frac{1}{j!}
\|\lie{\chi^{(r)}_{0}}^{j}
f^{(r-1,s-jr)}_{l+2j}\|_{1-d_{r-1}-\delta_r}\cr
&\leq\sum_{j=0}^{\lfloor s/r\rfloor}
\frac{1}{e}
\left(
\frac{\|X_0^{(r)}\|_{1-d_{r-1}}}{{\delta_r^2\rho\sigma}} + \frac{e |\zeta^{(r)}|}{{\delta_r\rho}}
\right)^{j}
 \nu_{r-1,s-jr} \Xi_r^{b(r-1,s-jr,l+2j)}
 \frac{E}{2^{l+2j}}\epsilon^{s-jr}\cr
 &\leq\sum_{j=0}^{\lfloor s/r\rfloor}
\frac{1}{e}
\left(
\frac{E}{{\omega\delta_r^2\rho\sigma}} + \frac{e E}{{4m\delta_r\rho}}
\right)^{j}
\nu_{r-1,r}^j \Xi_r^{b(r-1,r,2)j}\epsilon^{rj}\cr
&\qquad\qquad\qquad\times 
\nu_{r-1,s-jr} \Xi_r^{b(r-1,s-jr,l+2j)}
\frac{E}{2^{l+2j}}\epsilon^{s-jr}\cr&\leq 
\nu^{(\text{I})}_{r,s} \Xi_r^{b(\text{I};r-1,s,l)}  \frac{E}{2^l}\epsilon^s\ ,
\end{aligned}
$$
where we use the trivial inequality
$$
3rj-2j+b(r-1,s-jr,l+2j) \leq b(\rmI;r-1,s,l)\ .
$$

Coming to the second stage of the normalization step, just notice that the bound
for the generating function $\chi_2^{(r)}$ is similar to the one of
$\chi_0^{(r)}$ and in particular it has exactly the same exponent for the
coefficient $\Xi_r$.  Thus all the estimates appearing in the expansion of the
Hamiltonian $H^{(r)}$ are nothing but a minor {\it variazione}, mutatis
mutandis, with respect to the first stage of the normalization step.  This
concludes the proof of the Lemma.

\section{Proof of Proposition~\ref{prop:N-K}}
\label{app:fixedpoint}

The Proposition is a direct consequence of the Contraction Principle
applied to a suitable closed ball centered in $x_0$. Indeed, by
following a standard procedure (see, i.e., \cite{KolF89}), let us
formulate the original problem as a fixed point problem, namely
$F(x,\eps)=0$ if and only if $A(x,\eps)=x\,$, where
\begin{equation*}
  A(x,\eps) = x-[F'(x_0,\eps)]^{-1}F(x,\eps)\ .
\end{equation*}
We first of all show that $A$ is a contraction of a sufficiently small
ball centered in $x_0$. We first rewrite our assumptions in a more
general form
$$
\|F(x_0,\eps)\| \leq \mu\ ,\qquad
\|[F'(x_0,\eps)]^{-1}\|_{\Lscr(V)} \leq M\ ,
$$
and we introduce the auxiliary quantities
\begin{equation*}
  \eta= M\mu=C_1C_2|\eps|^{\beta-\alpha}\ ,\qquad\qquad h=M C_3\eta
  = C_1C_2^2C_3|\eps|^{\beta-2\alpha}\ .
\end{equation*}
Notice that the condition $\beta>2\alpha$ is necessary in order to have
\begin{displaymath}
  \lim_{\eps\to 0} h = 0\ .
\end{displaymath}
The main ingredient is the continuity of $F'$, since
$F\in\mathcal{C}^1$ locally around $x_0$ (independently from
$\eps$). From finite increment formula we get, for $x,y\in
B(x_0,r)\subset\Uscr(x_0)$
\begin{displaymath}
  \norm{A(x,\eps)-A(y,\eps)}\leq \tond{\sup_{z\in
      B(x_0,r)}\norm{A'(z,\eps)}_{\Lscr(V)}}\norm{x-y}\ ;
\end{displaymath}
thus, we aim at showing that, with a suitable choice of the radius
$r$, we have
\begin{displaymath}
  \sup_{z\in B(x_0,r)}\norm{A'(z,\eps)}_{\Lscr(V)}<1\ .
\end{displaymath}
Since
\begin{displaymath}
  A'(z,\eps) = \Id - \quadr{F'(x_0,\eps)}^{-1}F'(z,\eps) =
  \quadr{F'(x_0,\eps)}^{-1}\quadr{F'(x_0,\eps)-F'(z,\eps)}
\end{displaymath}
we get
$$\begin{aligned}
  \norm{A'(z,\eps)}_{\Lscr(V)}
  &\leq\norm{\quadr{F'(x_0,\eps)}^{-1}}_{\Lscr(V)}\norm{F'(x_0,\eps)-F'(z,\eps)}_{\Lscr(V)}\leq\\ &\leq
  M \norm{F'(x_0,\eps)-F'(z,\eps)}_{\Lscr(V)} \ .
\end{aligned}
$$ From the continuity of $F'$ it follows that, provided
$\norm{z-x_0}$ is small enough, it is possible to make
$F'(x_0,\eps)-F'(z,\eps)$ arbitrary small. The Lipschitz-continuity
estimate\footnote{Actually Holder-continuity will be sufficient, modifying the conditions on $\alpha$ and $\beta$.} in the
hypothesis of the Proposition allows to explicitly deal with this issue.
Indeed, from
\begin{displaymath}
  \norm{F'(x_0,\eps)-F'(z,\eps)}_{\Lscr(V)} \leq C_3\norm{z-x_0}\ ,
\end{displaymath}
we get
\begin{displaymath}
  \norm{A'(z,\eps)}_{\Lscr(V)}\leq MC_3\norm{z-x_0}\leq
  MC_3r =:q\ ,\qquad\qquad \forall z\in B(x_0,r)\ ,
\end{displaymath}
and also
\begin{displaymath}
\sup_{z\in B(x_0,r)}\norm{A'(z,\eps)}_{\Lscr(V)}\leq q\ .
  \end{displaymath}

In order to show that $F(B(x_0,r))\subset B(x_0,r)$, namely that
$ \norm{z-x_0}\leq r$ implies $\norm{A(z,\eps)-x_0}\leq r\,$,
we start splitting
\begin{displaymath}
  \norm{A(z,\eps)-x_0} \leq \norm{A(z,\eps)-A(x_0,\eps)} + \norm{A(x_0,\eps)-x_0}\ .
\end{displaymath}
We will separately estimate the two r.h.t.. From the bound on
$A'(z,\eps)$ we get
\begin{displaymath}
  \norm{A(z,\eps)-A(x_0,\eps)} \leq \sup_{z\in
    B(x_0,r)}\norm{A'(z,\eps)}_{\Lscr(V)}\norm{z-x_0}\leq
  qr\ .
\end{displaymath}
on the other hand, by exploiting the initial definition of
$A(x,\eps)$, one has
$$
\begin{aligned}
  \norm{A(x_0,\eps)-x_0} &= \norm{x_0 -
    [F'(x_0,\eps)]^{-1}F(x_0,\eps)-x_0} =
  \norm{[F'(x_0,\eps)]^{-1}F(x_0,\eps)} \leq \\ &\leq
  \norm{\quadr{F'(x_0,\eps)}^{-1}}_{\Lscr(V)}\norm{F(x_0,\eps)}\leq M\mu\ .
\end{aligned}
$$
Hence, in order to have $F(B(x_0,r))\subset B(x_0,r)$, it must happen
\begin{displaymath}
  M\mu + qr\leq r\ .
\end{displaymath}
Thus, two independent conditions have to be satisfied:
\begin{displaymath}
  MC_3r<1\ ,\qquad \eta + M C_3 r^2\leq r\ .
\end{displaymath}
The second is equivalent to
\begin{displaymath}
  M C_3 r^2 - r + \eta\leq 0\ ,
\end{displaymath}
which can be re-scaled to
\begin{displaymath}
  r=\eta\rho\ ,\qquad\qquad h\rho^2 -\rho+1\leq 0\ .
\end{displaymath}
The corresponding equation, under the condition $h<\frac14$, has the
two zeros
\begin{displaymath}
  t_\pm = \frac1{2h}\tond{1\pm\sqrt{1-4h}}\ .  
\end{displaymath}
Moreover one has $t_-<2\,$, since $1-4h<\sqrt{1-4h}\,$,
and for $h\sim 0$ we get $t_-(h)\sim 1\,$.
Collecting the above information, the radius $r$ has to fulfill
\begin{displaymath}
\eta t_-\leq r\leq t_+\eta\ .
\end{displaymath}
If we make the most restrictive choice
\begin{displaymath}
\eta t_-\leq r\leq 2\eta\ ,
\end{displaymath}
then, from $h<\frac14$, it follows that $F$ is an
$\frac12$-contraction map
\begin{displaymath}
  MC_3r<2MC_3\eta = 2h<\frac12\ .
\end{displaymath}
In our case, $h<\frac14$ comes directly from being $h(\eps)$
infinitesimal w.r.t. $\eps$; thus for $\eps$ small enough the
condition is satisfied. Moreover, from $h(\eps)\approx 1$, one deduces
that the optimal choice for the radius is
\begin{displaymath}
  r(\eps) = \eta t_- \approx C_1 C_2 |\eps|^{\beta-\alpha}\ .
\end{displaymath}

\qed

\begin{remark}
The above Proposition shows that $x_0$ is a better approximation of
the true solution as $\alpha$ decreases, which means as the
differential $F'(x_0,\eps)$ is bounded independently on $\eps$
\begin{displaymath}
  \norm{F'(x_0,\eps)}\geq C\qquad\Rightarrow\qquad \alpha=0\ .
\end{displaymath}
At the limiting case $\alpha=0$, it is possible to choose
$r=\Oscr(\eps^\beta)$.
\end{remark}


\section*{Acknowledgments}
We warmly thank Simone Paleari for useful discussions in the
preliminary part of this project, when trying to achieve the correct
normal form construction. T.P. also thanks his former Master student
Letizia Zanni, for her introductory work on Moser-Weinstein approach.
T.P. and M.S. have been partially supported by the National Group of
Mathematical Physics (GNFM-INdAM).

\def\cprime{$'$} \def\i{\ii}\def\cprime{$'$} \def\cprime{$'$}


\begin{thebibliography}{10}

\bibitem{Ahn98}
Taehoon Ahn.
\newblock Multisite oscillations in networks of weakly coupled autonomous
  oscillators.
\newblock {\em Nonlinearity}, 11(4):965--989, 1998.

\bibitem{Aub97}
Serge Aubry.
\newblock Breathers in nonlinear lattices: existence, linear stability and
  quantization.
\newblock {\em Phys. D}, 103(1-4):201--250, 1997.
\newblock Lattice dynamics (Paris, 1995).

\bibitem{BenGGS84}
G.~Benettin, L.~Galgani, A.~Giorgilli, and J.-M. Strelcyn.
\newblock A proof of {K}olmogorov's theorem on invariant tori using canonical
  transformations defined by the {L}ie method.
\newblock {\em Nuovo Cimento B (11)}, 79(2):201--223, 1984.

\bibitem{Ber07}
Massimiliano Berti.
\newblock {\em Nonlinear oscillations of {H}amiltonian {PDE}s}.
\newblock Progress in Nonlinear Differential Equations and their Applications,
  74. Birkh\"auser Boston, Inc., Boston, MA, 2007.

\bibitem{CorFG13}
Livia Corsi, Roberto Feola, and Guido Gentile.
\newblock Lower-dimensional invariant tori for perturbations of a class of
  non-convex {H}amiltonian functions.
\newblock {\em J. Stat. Phys.}, 150(1):156--180, 2013.

\bibitem{CorG15}
Livia Corsi and Guido Gentile.
\newblock Resonant motions in the presence of degeneracies for
  quasi-periodically perturbed systems.
\newblock {\em Ergodic Theory Dynam. Systems}, 35(4):1079--1140, 2015.

\bibitem{GioLoc-1997}
A.~Giorgilli and U.~Locatelli.
\newblock Kolmogorov theorem and classical perturbation theory.
\newblock {\em Zeitschrift f{\"u}r angewandte Mathematik und Physik ZAMP},
  48(2):220--261, Mar 1997.

\bibitem{GioSan-2012}
A.~{Giorgilli} and M.~{Sansottera}.
\newblock {Methods of algebraic manipulation in perturbation theory}.
\newblock {\em Workshop Series of the Asociacion Argentina de Astronomia},
  3:147--183, 2011.

\bibitem{Gio03}
Antonio Giorgilli.
\newblock Exponential stability of {H}amiltonian systems.
\newblock In {\em Dynamical systems. {P}art {I}}, Pubbl. Cent. Ric. Mat. Ennio
  Giorgi, pages 87--198. Scuola Norm. Sup., Pisa, 2003.

\bibitem{Giorgilli-2012}
Antonio Giorgilli.
\newblock On a theorem of lyapounov.
\newblock {\em Rend. Ist. Lomb. Acc. Sc. Lett.}, 146:133--160, 2012.

\bibitem{GioLocSan-2014}
Antonio Giorgilli, Ugo Locatelli, and Marco Sansottera.
\newblock On the convergence of an algorithm constructing the normal form for
  elliptic lower dimensional tori in planetary systems.
\newblock {\em Celestial Mech. Dynam. Astronom.}, 119(3-4):397--424, 2014.

\bibitem{Gro60}
Wolfgang Gr\"obner.
\newblock {\em Die {L}ie-{R}eihen und ihre {A}nwendungen}, volume~3 of {\em
  Mathematische Monographien}.
\newblock VEB Deutscher Verlag der Wissenschaften, Berlin, 1960.

\bibitem{HenT99}
D.~Hennig and G.~P. Tsironis.
\newblock Wave transmission in nonlinear lattices.
\newblock {\em Phys. Rep.}, 307(5-6):333--432, 1999.

\bibitem{Kev_book09}
Panayotis~G. Kevrekidis.
\newblock {\em The discrete nonlinear {S}chr\"odinger equation}, volume 232 of
  {\em Springer Tracts in Modern Physics}.
\newblock Springer-Verlag, Berlin, 2009.
\newblock Mathematical analysis, numerical computations and physical
  perspectives, Edited by Kevrekidis and with contributions by Ricardo
  Carretero-Gonz{\'a}lez, Alan R. Champneys, Jes{\'u}s Cuevas, Sergey V.
  Dmitriev, Dimitri J. Frantzeskakis, Ying-Ji He, Q. Enam Hoq, Avinash Khare,
  Kody J. H. Law, Boris A. Malomed, Thomas R. O. Melvin, Faustino Palmero,
  Mason A. Porter, Vassilis M. Rothos, Atanas Stefanov and Hadi Susanto.

\bibitem{KolF89}
A.~N. Kolmogorov and S.~V. Fomin.
\newblock {\em Elementy teorii funktsii i funktsionalnogo analiza}.
\newblock ``Nauka'', Moscow, sixth edition, 1989.
\newblock With a supplement, ``Banach algebras'', by V. M. Tikhomirov.

\bibitem{Kou13}
Vassilis Koukouloyannis.
\newblock Non-existence of phase-shift breathers in one-dimensional
  {K}lein-{G}ordon lattices with nearest-neighbor interactions.
\newblock {\em Phys. Lett. A}, 377(34-36):2022--2026, 2013.

\bibitem{KouK09}
Vassilis Koukouloyannis and Panayotis~G. Kevrekidis.
\newblock On the stability of multibreathers in {K}lein-{G}ordon chains.
\newblock {\em Nonlinearity}, 22(9):2269--2285, 2009.

\bibitem{MelS05}
E.~Meletlidou and G.~Stagika.
\newblock On the continuation of degenerate periodic orbits in {H}amiltonian
  systems.
\newblock {\em Regul. Chaotic Dyn.}, 11(1):131--138, 2006.

\bibitem{Mos76}
J{\"u}rgen~K. Moser.
\newblock Periodic orbits near an equilibrium and a theorem by {A}lan
  {W}einstein.
\newblock {\em Comm. Pure Appl. Math.}, 29(6):724--747, 1976.

\bibitem{PelKF05b}
D.~E. Pelinovsky, P.~G. Kevrekidis, and D.~J. Frantzeskakis.
\newblock Persistence and stability of discrete vortices in nonlinear
  {S}chr\"odinger lattices.
\newblock {\em Phys. D}, 212(1-2):20--53, 2005.

\bibitem{PelS12}
Dmitry Pelinovsky and Anton Sakovich.
\newblock Multi-site breathers in {K}lein-{G}ordon lattices: stability,
  resonances and bifurcations.
\newblock {\em Nonlinearity}, 25(12):3423--3451, 2012.

\bibitem{PenKSKP17}
T.~Penati, V.~Koukouloyannis, M.~Sansottera, P.G. Kevrekidis, and S.~Paleari.
\newblock On the (non) existence of degenerate phase-shift multibreathers in a
  zog-zag klein-gordon model via dnls approximation.
\newblock {\em preprint}, 2017.

\bibitem{PenSPKK16}
T.~Penati, M.~Sansottera, S.~Paleari, V.~Koukouloyannis, and P.G. Kevrekidis.
\newblock On the (non) existence of phase-shift discrete solitons in a dnls
  nonlocal lattice.
\newblock {\em ArXiv e-prints}, 2017.
\newblock 1707.01679.

\bibitem{PoiI}
H.~Poincar\'e.
\newblock {\em Les m\'ethodes nouvelles de la m\'ecanique c\'eleste. {T}ome
  {I}. {S}olutions p\'eriodiques. {N}on-existence des int\'egrales uniformes.
  {S}olutions asymptotiques}.
\newblock Dover Publications, Inc., New York, N.Y., 1957.

\bibitem{PoiOU7}
Henri Poincar\'e.
\newblock {\em \OE uvres. {T}ome {VII}}.
\newblock Les Grands Classiques Gauthier-Villars. [Gauthier-Villars Great
  Classics]. \'Editions Jacques Gabay, Sceaux, 1996.
\newblock Masses fluides en rotation. Principes de m\'ecanique analytique.
  Probl\`eme des trois corps. [Rotating fluid masses. Principles of analytic
  mechanics. Three-body problem], With a preface by Jacques L\'evy, Reprint of
  the 1952 edition.

\bibitem{SanCec-2017}
Marco Sansottera and Marta Ceccaroni.
\newblock Rigorous estimates for the relegation algorithm.
\newblock {\em Celestial Mechanics and Dynamical Astronomy}, 127(1):1--18, Jan
  2017.

\bibitem{SanLocGio-2010}
Marco Sansottera, Ugo Locatelli, and Antonio Giorgilli.
\newblock A semi-analytic algorithm for constructing lower dimensional elliptic
  tori in planetary systems.
\newblock {\em Celestial Mech. Dynam. Astronom.}, 111(3):337--361, 2011.

\bibitem{VoyI99}
George Voyatzis and Simos Ichtiaroglou.
\newblock Degenerate bifurcations of resonant tori in {H}amiltonian systems.
\newblock {\em Internat. J. Bifur. Chaos Appl. Sci. Engrg.}, 9(5):849--863,
  1999.

\bibitem{Wei73b}
Alan Weinstein.
\newblock Normal modes for nonlinear {H}amiltonian systems.
\newblock {\em Invent. Math.}, 20:47--57, 1973.

\bibitem{YakS75}
V.~A. Yakubovich and V.~M. Starzhinskii.
\newblock {\em Linear differential equations with periodic coefficients. 1, 2}.
\newblock Halsted Press [John Wiley \& Sons]\ New York-Toronto, Ont.,; Israel
  Program for Scientific Translations, Jerusalem-London, 1975.
\newblock Translated from Russian by D. Louvish.

\end{thebibliography}
\end{document}